\documentclass[12pt]{amsart}

\usepackage[colorlinks=true,linkcolor=blue]{hyperref}
\usepackage{amsmath}
\usepackage{amssymb}
\usepackage{amsthm}

\usepackage{graphicx}
\usepackage{epstopdf}
\usepackage{epsfig}

\usepackage{enumitem}

\usepackage{hyperref}

\usepackage{times}
\usepackage{relsize}
\usepackage{textcomp}
\usepackage{amssymb}
\usepackage[english]{babel}
\usepackage[autostyle]{csquotes}
\usepackage{epstopdf}
\usepackage{nicefrac}
\usepackage{tikz}

\theoremstyle{plain} 
\newtheorem{thm}{Theorem}[section]
\newtheorem{prop}[thm]{Proposition}
\newtheorem{lemma}[thm]{Lemma}
\newtheorem{cor}[thm]{Corollary} 
\newtheorem{question}[thm]{Question}
\newtheorem{conj}[thm]{Conjecture}
\theoremstyle{remark}
\newtheorem{remark}[thm]{Remark}

\theoremstyle{definition}
\newtheorem{defin}[thm]{Definition}

\newcommand{\CC}{\mathbb{C}}

\newcommand{\NN}{\mathbb{N}}

\newcommand{\PP}{\mathbb{P}}
\newcommand{\QQ}{\mathbb{Q}}

\newcommand{\ZZ}{\mathbb{Z}}

\newcommand{\calO}{{\mathcal O}}

\DeclareMathOperator{\Orb}{Orb}

\DeclareMathOperator{\Res}{Res}

\newcommand{\frp}{\mathfrak{p} }

\newcommand{\la}{\langle}
\newcommand{\ra}{\rangle}

\newcommand{\dsps}{\displaystyle}

\begin{document}

\title{Misiurewicz polynomials and dynamical units, part II}
\date{May 19, 2022}
\author[Benedetto]{Robert L. Benedetto}
\address{Amherst College \\ Amherst, MA 01002 \\ USA}
\email{rlbenedetto@amherst.edu}
\author[Goksel]{Vefa Goksel}
\address{University of Massachusetts \\ Amherst, MA 01002 \\ USA}
\email{goksel@math.umass.edu}

\begin{abstract}
Fix an integer $d\geq 2$.
The parameters $c_0\in \overline{\mathbb{Q}}$ for which the unicritical polynomial 
$f_{d,c}(z)=z^d+c\in \mathbb{C}[z]$ has finite postcritical orbit, also known
as \emph{Misiurewicz} parameters, play a significant role in complex dynamics.
Recent work of Buff, Epstein, and Koch proved the first known cases of
a long-standing dynamical conjecture of Milnor using their arithmetic properties,
about which relatively little is otherwise known.
Continuing our work in a companion paper, we address further
arithmetic properties of Misiurewicz parameters, especially the
nature of the algebraic integers obtained by evaluating the polynomial
defining one such parameter at a different Misiurewicz parameter.
In the most challenging such combinations,
we describe a connection between such algebraic integers
and the multipliers of associated periodic points.
As part of our considerations,
we also introduce a new class of polynomials we call \emph{$p$-special},
which may be of independent number theoretic interest.
\end{abstract}

\maketitle

\section{Introduction}
Let $f\in \CC(z)$ be a rational function.
For each integer $n\geq 0$, we write $f^n$ for the $n$-th iterate of $f$ under composition,
i.e., $f^0(z):=z$, and $f^n:=f\circ f^{n-1}$ for each $n\geq 1$.

A point $x\in\PP^1(\CC)$ is \emph{periodic} (of period $n$)
if there is an integer $n\geq 1$ such that $f^n(x)=x$.
In that case, the smallest such integer is the \emph{exact} period of $x$,
and if $x\neq \infty$, then the \emph{multiplier} $\lambda\in \CC$ of $x$ is
\begin{equation}
\label{eq:multdef}
\lambda:=\big(f^n\big)'(x) = \prod_{i=0}^{n-1} f'\big(f^i(x)\big) .
\end{equation}
(One can also define the multiplier of a periodic point at $\infty$ via coordinate change.)

More generally, we say $x\in \mathbb{P}^1(\mathbb{C})$ is \emph{preperiodic} if there is some $m\geq 0$
such that $f^m(x)$ is periodic. That is, $x$ is preperiodic
if and only if its (\emph{forward}) \emph{orbit}
\[\Orb_f^+(x) := \{f^n(x) : n\geq 0\} \]
is finite.
In that case, the \emph{tail length} of $x$ is the smallest integer $m\geq 0$ such that
$f^m(x)$ is periodic.
We say $x$ is \emph{preperiodic of type} $(m,n)$ if $x$ is preperiodic with tail length $m$,
and $n$ is the exact period of $f^m(x)$. Equivalently,
$x$ is preperiodic of type $(m,n)$ if $m\geq 0$ is the minimal nonnegative integer
and $n\geq 1$ is the minimal positive integer such that $f^{m+n}(x)=f^m(x)$.

We call $f\in \mathbb{C}(z)$ \emph{postcritically finite} (or \emph{PCF})
if all of its critical points are preperiodic.

In this paper we consider polynomials in the unicritical family
$f_{d,c}:=z^d+c\in \CC[z]$, whose only critical points are $\infty$ and $0$.
Since $\infty$ is a fixed point, the polynomial $f_{d,c}$ is postcritically finite if and only if the
forward orbit
\[\{0,c,c^d+c,(c^d+c)^d+c,\dots\}\]
of the critical point $0$ is finite.
We set the notation $a_1=c$, and $a_{i+1}=a_i^2+c$ for $i\geq 1$,
where we consider each $a_i$ as an element of the polynomial ring $\ZZ[c]$.

Any parameter $c_0$ for which $f_{d,c_0}$ is postcritically finite is an algebraic integer.
Indeed, as we described in Section~1 of \cite{BG1},
if $0$ is not in the post-critical orbit of $f_{d,c_0}$, then $c_0$ is a root of a monic polynomial
$G_{d,m,n}^{\zeta}\in \mathbb{Z}[\zeta][c]$ for some $m\geq 2, n\geq 1$,
where $\zeta\neq 1$ is a $d$-th root of unity. Here, the pair $(m,n)$ is the preperiodic type
of the parameter $c_0$, as defined above,
and $\zeta\neq 1$ specifies the relation $\zeta a_{m-1}(c_0) = a_{m+n-1}(c_0)$.
The polynomials $G_{d,m,n}^{\zeta}$ are called \emph{$(m,n)$-Misiurewicz polynomials},
and they are defined by
\begin{equation}
\label{eq:Gdef}
G_{d,m,n}^{\zeta}(c):= \prod_{k|n} \big( a_{m+k-1} - \zeta a_{m-1}\big)^{\mu(n/k)}
\cdot \begin{cases}
\prod_{k|n} \big( a_{k}\big)^{-\mu(n/k)}
& \text{ if } n|m-1,
\\
1 & \text{ if } n \nmid m-1.
\end{cases}
\end{equation}
For $d=2$, the $d$-th root of unity $\zeta\neq 1$ is necessarily $-1$. For this reason,
we sometimes write simply $G_{m,n}$ instead of $G_{d,m,n}^{\zeta}$ in this case.

Misiurewicz parameters have been extensively studied in complex dynamics, especially in the quadratic case. Here are just a few examples of known results.
Douady and Hubbard \cite[Chapter 8]{DH} proved that Misiurewicz parameters are dense
in the boundary of the Mandelbrot set.
Poirer \cite{Poirier93} used Hubbard trees to give a classification of the dynamics
of PCF polynomials (see Theorems A and B in \cite{Poirier93}). Eberlein \cite{Eber99} showed that the periodic cycle in the
post-critical orbit of $f_{d,c_0}$ for any Misiurewicz parameter $c_0$ is repelling.
Favre and Gauthier \cite[Theorem 1]{FG15} strengthened Douady and Hubbard's density result
by proving that Misiurewicz parameters are equidistributed in the boundary of
the Mandelbrot set; see also \cite{GKNY17}.

Given their arithmetic nature, Misiurewicz polynomials have drawn number theoretic
interest as well. For example, in Theorems 1.1 and 1.3 of \cite{Fakh14}, Fakhruddin used the simplicity of roots'
of Misiurewicz polynomials to prove
dynamical analogues of the Mordell-Lang and Manin-Mumford conjectures for generic
endomorphisms of $\PP^n$.
In \cite{BD13}, Baker and DeMarco presented a dynamical analogue of the
Andr\'{e}-Oort Conjecture by considering PCF parameters (such as Misiurewicz parameters)
in dynamical moduli spaces to Andr\'{e}-Oort special points on Shimura varieties.
That is, PCF parameters in dynamical moduli spaces are analogous
to CM points on modular curves, and they are expected to have
correspondingly analogous arithmetic properties.
See, for example, \cite{GKNY17} for more on the dynamical Andr\'{e}-Oort conjecture.

One key open question from both the complex dynamical and number theoretic
perspectives is the following variant of a conjecture of Milnor from \cite{Milnor12}.

\begin{conj}
\label{conj:Misiurewicz_irr}
Let $d,m\geq 2, n\geq 1$. Suppose that $\zeta\neq 1$ is a $d$-th root of unity.
Then $G_{d,m,n}^{\zeta}$ is irreducible over $\QQ(\zeta)$.
\end{conj} 

Besides some limited computational evidence and partial results, very little is known about Conjecture~\ref{conj:Misiurewicz_irr}. The irreducibility is known, for instance, in the cases
that $d=2$ and $n\leq 3$, that $d=3$ and $n=2$, and that $d$ is a prime power and $n=1$.
See \cite{BEK19,Gok19,Gok20} for these results and more.
In particular, in Theorems~1 and~4 of \cite{BEK19},
Buff, Epstein, and Koch used some
special cases of Conjecture~\ref{conj:Misiurewicz_irr}
proven by the second author in \cite[Corollary 1.1]{Gok19}
to prove the first known cases of a different conjecture of Milnor \cite{Milnor93,Milnor09}
on the irreducibility of certain curves arising as dynamical moduli spaces.
 
Although Conjecture~\ref{conj:Misiurewicz_irr} appears to be currently out of reach,
it is but one piece of the broader question that always accompanies arithmetically interesting
families of polynomials: what are the properties of the number fields generated by their roots?
Inspired by parallels with cyclotomic and elliptic units,
in \cite{BG1}, we posed the following question:
\begin{question}
\label{ques:unit}
Fix $d,m\geq 2$, $n\geq 1$, and $\zeta\neq 1$ a $d$-th root of unity.
Let $c_0$ be a root of $G_{d,m,n}^\zeta$, and
let $K:=\QQ(c_0)$.
For which integers
$j\geq 2$ and $\ell\geq 1$ is $G_{d,j,\ell}^{\zeta}(c_0)$ an algebraic unit in $\calO_K$?
\end{question}
Similar questions arose
in \cite{BEK19,Gok19} surrounding \emph{Gleason polynomials}, which are analogues
of Misiurewicz polynomials for the case that the critical point is periodic rather than
strictly preperiodic. In particular, in Lemma 3.1 of \cite{Gok19}, the second author proved that
evaluating one Gleason polynomial at the root of another yields an algebraic unit.
On the other hand, Buff, Epstein, and Koch \cite[Lemma 26]{BEK19} considered resultants of Misiurewicz polynomials with Gleason polynomials, and they proved that evaluating a Misiurewicz polynomial at a Gleason parameter gives an algebraic unit unless the periods of these polynomials match.
They then leveraged this result to prove Misiurewicz irreducibility results towards
Conjecture~\ref{conj:Misiurewicz_irr}.

In \cite{BG1}, when $d=p^e$ is a prime power, we gave a complete answer
to Question~\ref{ques:unit} if $j\neq m$.
However, the much harder case seems to be when $j=m$, for which we posed
the following conjecture, based on Magma computations for small values of $m$ and $n$.
\begin{conj}
\label{conj:j=m}
Let $d=p^e$, where $p$ is a prime and $e\geq 1$. Let $c_0$ be a root of $G_{d,m,n}^\zeta$
for some $m\geq 2$, $n\geq 1$, and $\zeta\neq 1$ a $d$-th root of unity.
Set $K:=\QQ(c_0)$. Suppose that $1\leq\ell\leq n$. Then
\[ G_{d,m,\ell}^{\zeta}(c_0) \text{ is a unit in } \calO_K
\quad\text{ if and only if }\quad
\ell\nmid n.\]
\end{conj}

When $\ell=n$,  we have $G_{d,m,\ell}^{\zeta}(c_0)=0$, which is not a unit,
and hence Conjecture~\ref{conj:j=m} holds trivially.
Thus, the conjecture is immediate for $n=1$,
and hence we will often assume $n\geq 2$.
In addition, 
when considering the forward implication of Conjecture~\ref{conj:j=m},
we can restrict our attention to the case that $\ell$ is a \emph{proper} divisor of $n$.

The main results of the current paper are as follows.
We prove the reverse implication of Conjecture~\ref{conj:j=m} in
Proposition~\ref{prop:unit-if} of Section~\ref{sec:reverse}.
We also describe a connection between the forward implication
and certain arithmetic properties of the multiplier of the periodic cycle in the postcritical orbit
$\Orb_f^{+}(c_0)$. Namely, for $m\geq 2, n\geq 1$, we define the
\emph{multiplier polynomial} $P_{d,m,n}^{\zeta}\in \mathbb{Z}[\zeta][x]$
to be the monic polynomial
whose roots are the multipliers that correspond to Misiurewicz parameters of type $(m,n)$.
In the case $d=2$, assuming the irreducibility of $G_{m,n}$ over $\mathbb{Q}$,
we reduce Conjecture~\ref{conj:j=m} to the following conjecture,
which relates our question to classical cyclotomic polynomials.
\begin{conj}
\label{conj:resultant_mult_cycl}
Let $m\geq 2, n,\ell\geq 1$. Then $|\Res(P_{m,n}, \Phi_{\ell})|>1$. 
\end{conj}
Magma computations suggest that the resultants of Conjecture~\ref{conj:resultant_mult_cycl}
are huge integers. In particular, Table~\ref{tab:logres} in Section~\ref{sec:appendix}
gives values of $\log|\Res(P_{m,n}, \Phi_{\ell})|$
for $m+n\leq 7$ and $\ell\leq 8$, illustrating that these values are very large.
Nevertheless, the relatively modest claim of
Conjecture~\ref{conj:resultant_mult_cycl}, merely that these resultants are
greater than 1, appears to be quite difficult.
For fixed $m$ and $n$, one can prove such a result
for sufficiently large $\ell$ using Diophantine approximation methods
(see, for example, \cite{Kam88}), but we need the statement for \emph{all} $\ell\geq 1$.

Another obstacle one faces in studying Conjecture~\ref{conj:resultant_mult_cycl} is that
we do not yet have a clean, explicit definition for the multiplier polynomials $P_{m,n}$,
as we do for Misiurewicz and Gleason polynomials. To overcome this difficulty, we
have introduced a broader class of polynomials that we conjecture the multiplier
polynomials $P_{m,n}$ may belong to, and
which may be of independent number theoretic interest.
\begin{defin}
\label{def:p-special_intro}
Let $P(x)=x^i+A_{i-1}x^{i-1}+\dots+A_1x+A_0\in\ZZ[x]$ be a monic polynomial with integer coefficients,
and let $p$ be a prime number.
We say that $P(x)$ is \emph{p-special} if it satisfies the following two properties:
\begin{itemize}
	\item $v_p(A_{i-1})> v_p(2)$, and
	\item $v_p(A_j)>v_p(A_{i-1})$ for $j=0,1,\dots,i-2$.
\end{itemize}
\end{defin}
In particular, we prove the following result about $p$-special polynomials,
intended as a stepping stone towards Conjecture~\ref{conj:resultant_mult_cycl}.
\begin{thm}
\label{thm:resultant>1_intro}
Let $p$ be a prime number, and let $\ell \geq 1$.
For any $p$-special polynomial $f(x)\in \mathbb{Z}[x]$,
we have $|\Res(f,\Phi_{\ell})|>1$.
\end{thm}
Thus, if the multiplier polynomials are $2$-special, and if Misiurewicz polynomials
are irreducible as in Conjecture~\ref{conj:Misiurewicz_irr}, then
Theorem~\ref{thm:resultant>1_intro} would imply Conjecture~\ref{conj:j=m}.
In particular, we are able
to establish $2$-specialness in the cases $n=1,2$, yielding the following result.
\begin{thm}
\label{thm:d=n=2}
Let $d=2$ and $m,n\geq 2$. Assume that $G_{m,n}$ is irreducible over $\QQ$.
Let $c_0$ be a root of $G_{m,n}$, and let $K:=\QQ(c_0)$. Then:
\begin{enumerate}[label=\textup{(\arabic*)}]
\item $G_{m,1}(c_0)$ is not a unit in $\calO_K$.
\item If $n$ is even, then $G_{m,2}(c_0)$ is not a unit in $\calO_K$.
\end{enumerate}
\end{thm}
Because $G_{m,n}$ is already known to be irreducible for $n\leq 3$
(see \cite{BEK19},\cite{Gok20}), we have the following unconditional corollary.
\begin{cor}
Let $d=2$ and $m\geq 2$. Conjecture~\ref{conj:j=m} holds for $n\leq 3$.
\end{cor}

The structure of the paper is as follows.
In Section~\ref{sec:reverse},
we prove the \emph{if} direction of Conjecture~\ref{conj:j=m}.
In Section~\ref{sec:techlem}, we prove some technical lemmas
on the ideals generated by certain values of \emph{Misiurewicz polynomials}.
Using these lemmas, we relate those Misiurewicz values to multipliers
of the associated periodic cycles in Section~\ref{sec:mult}.
We then define polynomials $P_{d,m,n}^{\zeta}$
defining these multipliers in Section~\ref{sec:multpoly}
and relate Conjecture~\ref{conj:j=m} to the resultants of these multiplier polynomials
with cyclotomic polynomials. 
In Section~\ref{sec:special}, we introduce the class of polynomials we dub $p$-special,
and in Section~\ref{sec:imply}, we conjecture that when $d=2$,
our multiplier polynomials are 2-special. In Section~\ref{sec:finish}, we prove this conjecture about multiplier polynomials for $d=2$ and $n=1,2$,
and we use this fact to deduce Theorem~\ref{thm:d=n=2}. Finally, in Section~\ref{sec:appendix}, we provide some empirical data related to Conjecture~\ref{conj:resultant_mult_cycl}.


\section{The reverse implication}
\label{sec:reverse}
In this section, 
we prove the easier \emph{if} direction of Conjecture~\ref{conj:j=m}.
\begin{prop}
\label{prop:unit-if}
Let $d,m,n\geq 2$. Let $c_0$ be a root of $G_{d,m,n}^{\zeta}$, where $\zeta\neq 1$ is a $d$-th root of unity.
Set $K=\QQ(c_0)$. For $1\leq \ell\leq n$, if $\ell\nmid n$, then $G_{d,m,\ell}^{\zeta}(c_0)$
is an algebraic unit in $\calO_K$.
\end{prop}

To prove Proposition~\ref{prop:unit-if}, we will need the following two lemmas.

\begin{lemma}
\label{lem:samev}
Let $d,m,n\geq 2$. Let $c_0$ be a root of $G_{d,m,n}^{\zeta}$, where $\zeta\neq 1$ is a $d$-th root of unity.
Set $K=\QQ(c_0)$.
For any positive integers $u,v \geq 1$ with $\gcd(u,n)=\gcd(v,n)$, we have
\[\big\la a_{m+u-1}(c_0)-\zeta a_{m-1}(c_0)\big\ra=\big\la a_{m+v-1}(c_0)-\zeta a_{m-1}(c_0)\big\ra.\]
as ideals in $\calO_K$.
\end{lemma}

\begin{proof}
For every finite place $\frp$ of $\calO_K$, we will establish the equality
\begin{equation}
\label{eq:samev}
v_{\frp}\big(a_{m+u-1}(c_0)-\zeta a_{m-1}(c_0)\big)=v_{\frp}\big(a_{m+\gcd(u,n)-1}(c_0)-\zeta a_{m-1}(c_0)\big)
\end{equation}
for every positive integer $u \geq 1$, from which the lemma follows immediately.

Fix such an integer $u$, and consider an arbitrary finite place $\frp$ of $\calO_K$.

Suppose first that $a_{m+u-1}(c_0)\equiv \zeta a_{m-1}(c_0) \pmod{\frp^i}$ for some integer $i\geq 1$.
Applying $u$ iterations of $f:=f_{d,c_0}$ to this congruence, we obtain
\[a_{m+2u-1}(c_0)\equiv a_{m+u-1}(c_0)\equiv \zeta a_{m-1}(c_0)\pmod{\frp^i}.\]
Thus, proceeding inductively, we have $a_{m+k u-1}(c_0)\equiv \zeta a_{m-1}(c_0)\pmod{\frp^i}$
for any integer $k\geq 1$.
Since $f$ has exact type $(m,n)$, it follows that
$a_{m+ku+\ell n-1}(c_0)\equiv \zeta a_{m-1}(c_0) \pmod{\frp^i}$ for any integers $k,\ell\geq 1$.
It is possible to choose positive integers $k,\ell, t\geq 1$ such that $k u+\ell n=\gcd(u,n)+nt$,
and hence we have $a_{m+\gcd(u,n)-1}(c_0)\equiv \zeta a_{m-1}(c_0) \pmod{\frp^i}$.
Thus, we have proven the $\leq$ direction of equation~\eqref{eq:samev}.

Conversely, suppose that $a_{m+\gcd(u,n)-1}(c_0)\equiv\zeta a_{m-1}(c_0)\pmod{\frp^i}$
for some integer $i\geq 1$.
By a similar inductive argument as above, we have
$a_{m+t\gcd(u,n)-1}(c_0)\equiv \zeta a_{m-1}(c_0)\pmod{\frp^i}$ for any integer $t\geq 1$.
In particular, therefore, we have $a_{m+u-1}(c_0)\equiv \zeta a_{m-1}(c_0)\pmod{\frp^i}$,
proving the $\geq$ direction of equation~\eqref{eq:samev}.
\end{proof}

\begin{remark}
\label{rem:ndivu}
If $\gcd(u,n)=\gcd(v,n)=n$, i.e., if $n|u$ and $n|v$,
then the ideal in the statement of Lemma~\ref{lem:samev} is the zero ideal.
After all, in that case, we have
\[ a_{m+u-1}(c_0) = a_{m+n-1}(c_0)=\zeta a_{m-1}(c_0), \]
because $c_0$ is a root of $G_{d,m,n}^{\zeta}$.

On the other hand, if $n\nmid u$, i.e., if $\ell:=\gcd(u,n)$ satisfies $1\leq \ell < n$,
then $a_{m+u-1}(c_0) \neq \zeta a_{m-1}(c_0)$, by Lemma~2.2 of \cite{BG1},
and hence the ideal in Lemma~\ref{lem:samev} above is nonzero.
\end{remark}

\begin{lemma}
\label{lem:elementary}
Let $\ell,n\geq 1$ be positive integers with $\ell\nmid n$. Then, for any positive integer $t|n$, we have
\[\sum_{\substack{k|\ell\\ \gcd(k,n)=t}} \mu\bigg(\frac{\ell}{k}\bigg)=0.\]
\end{lemma}

\begin{proof}
For each integer $k$ in the sum, we may write $k=tk_1$ and $n=tn_1$,
where $k_1,n_1\geq 1$ are relatively prime positive integers.
Moreover, because $k|\ell$, we may write $\ell=t\ell_1$ for some positive integer $\ell_1\geq 1$
such that $\ell_1\nmid n_1$. Thus,
\[\sum_{\substack{k|\ell\\ \gcd(k,n)=t}} \mu\bigg(\frac{\ell}{k}\bigg)=
\sum_{\substack{k_1|\ell_1 \\ \gcd(k_1,n_1)=1}} \mu\bigg(\frac{\ell_1}{k_1}\bigg). \]
Hence, it suffices to prove the lemma in the case that $t=1$, which we assume hereafter.

Define a function $v_n:\NN\to\NN$ by
\[ v_n(\ell) = \begin{cases}
1 & \text{ if }  \text{gcd}(\ell,n)=1\\
0 & \text{ if }  \text{gcd}(\ell,n)>1.
\end{cases} \]
It is straightforward to check that $v_n$ is multiplicative.
Therefore, being the Dirichlet convolution of two multiplicative functions, the function
\[F_n(\ell):=\sum_{k|\ell}\mu\bigg(\frac{\ell}{k}\bigg)v_n(k)
= \sum_{\substack{k|\ell \\ \gcd(k,n)=1}} \mu\bigg(\frac{\ell}{k}\bigg) \]
is also multiplicative. We wish to show that $F_n(\ell)=0$ for all $\ell\in\NN$
with $\ell\nmid n$. 
Since $F_n$ is multiplicative, it suffices to show that $F_n(p^e)=0$
for any prime $p$ and any integer $e\geq 1$ such that $p^e\nmid n$.

Given such $p$ and $e$, we consider two cases.
In the first case, suppose that $p|n$.
Then because $p^e\nmid n$, we must have $e\geq 2$.
In addition, the only positive integer $k$ such that $k|p^e$ and $\gcd(k,n)=1$ is $k=1$.
Thus,
\[F_n(p^e) = \sum_{\substack{k|p^e \\ \gcd(k,n)=1}} \mu\bigg(\frac{p^e}{k}\bigg) = \mu(p^e) = 0,\]
where the final equality is because $e\geq 2$.

The only other case is that $p\nmid n$. In this case, all divisors of $p^e$ are relatively prime to $n$,
and hence
\[F_n(p^e) = \sum_{\substack{k|p^e \\ \gcd(k,n)=1}} \mu\bigg(\frac{p^e}{k}\bigg) 
= \sum_{k|p^e} \mu\bigg(\frac{p^e}{k}\bigg) = \mu(p)+\mu(1)=0, \]
where the third equality is because $\mu(p^i)=0$ for $i\geq 2$.
\end{proof}

\begin{proof}[Proof of Proposition~\ref{prop:unit-if}]
By the M\"{o}bius product definition of $G_{d,m,\ell}^{\zeta}$, we have
\[ \big\la G_{d,m,\ell}^{\zeta}(c_0)\big\ra \bigg|
\prod_{k|\ell} \big\la a_{m+k-1}(c_0)-\zeta a_{m-1}(c_0)\big\ra^{\mu(\ell/k)}.\]
Therefore, it suffices to show that
\begin{equation}
\label{eq:goalprop}
\prod_{k|\ell} \big\la a_{m+k-1}(c_0)-\zeta a_{m-1}(c_0)\big\ra^{\mu(\ell/k)}=\calO_K.
\end{equation}
To this end, the product in equation~\eqref{eq:goalprop} is
\begin{align*}
\prod_{t|n} \prod_{\substack{k|\ell\\ \gcd(k,n)=t}} & \big\la a_{m+k-1}(c_0)-\zeta a_{m-1}(c_0)\big\ra^{\mu(\ell/k)}
= \prod_{t|n} \prod_{\substack{k|\ell\\ \gcd(k,n)=t}}\big\la a_{m+t-1}(c_0)-\zeta a_{m-1}(c_0)\big\ra^{\mu(\ell/k)}
\\
&= \prod_{t|n} \big\la a_{m+t-1}(c_0)-\zeta a_{m-1}(c_0)\big\ra^{E_t},
\quad\text{where}\quad
E_t:=\sum_{\substack{k|\ell\\ \gcd(k,n)=t}} \mu\bigg(\frac{\ell}{k}\bigg) ,
\end{align*}
and where we have applied Lemma~\ref{lem:samev} in the first equality above.
However, we have $E_t=0$ for all $t|n$, by Lemma~\ref{lem:elementary},
so that the product is simply $\calO_K$, as desired.
\end{proof}


\section{Technical Lemmas}
\label{sec:techlem}
The rest of this paper is devoted to the \emph{only if} direction of Conjecture~\ref{conj:j=m},
which is much more involved than the \emph{if} direction.
To that end, we present two auxiliary lemmas in this section.

\begin{lemma}
\label{lem:5.7}
Let $d,m,n\geq 2$. Let $c_0$ be a root of $G_{d,m,n}^{\zeta}$, where $\zeta\neq 1$ is a $d$-th root of unity,
and let $h\in\ZZ[\zeta][c]$ be the minimal polynomial of $c_0$ over $\QQ(\zeta)$.
Let $1\leq\ell<n$, and let $\alpha_0$ be a root of $G_{d,m,\ell}^{\zeta}$.
Set $K=\QQ(c_0)$ and $L=\QQ(\alpha_0)$.
If $h(\alpha_0)$ is not a unit in $\calO_L$, then $G_{d,m,\ell}^{\zeta}(c_0)$ is not a unit in $\calO_K$.
\end{lemma}

\begin{proof}
Let $g\in \ZZ[\zeta][c]$ be the minimal polynomial of $G_{d,m,\ell}^{\zeta}(c_0)$ over $\QQ(\zeta)$.
We need to show that the norm $\text{N}_{K/\QQ(\zeta)}(G_{d,m,\ell}^{\zeta}(c_0))$
is not a unit in $\ZZ[\zeta]$, or equivalently, that
$g(0)$ is not a unit in $\ZZ[\zeta]$.

By definition of $g$, we have $(g\circ G_{d,m,\ell}^{\zeta})(c_0)=0$.
Since $h$ is the minimal polynomial of $c_0$ over $\QQ(\zeta)$,
there is a polynomial $u\in\ZZ[\zeta][c]$ such that $g\circ G_{d,m,\ell}^{\zeta}=h\cdot u$.
Evaluating at $\alpha_0$ and recalling that $G_{d,m,\ell}^{\zeta}(\alpha_0)=0$, we obtain
\[g(0) = h(\alpha_0)u(\alpha_0).\]
By hypothesis, $h(\alpha_0)$ is not a unit in $\calO_L$.
Since $u(\alpha_0)$ is also an algebraic integer, it follows that $g(0)=h(\alpha_0)u(\alpha_0)$
is also not a unit in $\calO_L$.
Therefore $g(0)$ is not a unit in $\calO_L\cap \QQ(\zeta) = \ZZ[\zeta]$, as desired.
\end{proof}

\begin{remark}
\label{rem:alpha_0_unit}
With notation as above, if $G_{d,m,n}^{\zeta}$ is irreducible over $\QQ(\zeta)$, then
Lemma~\ref{lem:5.7} says that if $G_{d,m,n}^{\zeta}(\alpha_0)$ is not a unit in $\calO_L$,
then $G_{d,m,\ell}^{\zeta}(c_0)$ is not a unit in $\calO_K$.
\end{remark}

Motivated by Remark~\ref{rem:alpha_0_unit},
we turn our attention to the algebraic integers $G_{d,m,n}^{\zeta}(\alpha_0)$,
where $\alpha_0$ is a root of $G_{d,m,\ell}^{\zeta}$ for some proper divisor $\ell$ of $n$.
To prove the \emph{only if} direction of Conjecture~\ref{conj:j=m},
it suffices to show that these algebraic integers are not algebraic units.

Fix integers $d,m,n\geq 2$, and let $\zeta\neq 1$ be a $d$-th root of unity.
Let $c_0$ be a root  of $G_{d,m,n}^{\zeta}$,
fix a positive integer $\ell$ that is a \emph{proper} divisor of $n$,
and let $\alpha_0$ be a root of $G_{d,m,\ell}^{\zeta}$.
Define sequences $\{B_j\}_{j\geq 1}$ and $\{b_j\}_{j\geq 1}$ in the polynomial ring $\ZZ[\zeta,c]$ by
\begin{equation}
\label{eq:Bdef}
B_j:= a_{m+\ell j-1} - \zeta a_{m-1}
\quad\text{and}\quad
b_j := \frac{B_j}{B_1} = \frac{a_{m+\ell j-1} - \zeta a_{m-1}}{a_{m+\ell-1} - \zeta a_{m-1}}.
\end{equation}
To see that $b_j$ is indeed a polynomial in $\ZZ[\zeta,c]$,
both its numerator $a_{m+\ell j-1} - \zeta a_{m-1}$ and denominator $a_{m+\ell-1} - \zeta a_{m-1}$
are monic polynomials in $\ZZ[\zeta,c]$, and as shown in the proof of Theorem~A.1 of \cite{Epstein12},
they both have only simple roots. Moreover, because $\ell | \ell j$, Lemma~2.2 of \cite{BG1}
shows that every root of the denominator is also a root of the numerator.
Thus, as claimed, the quotient $b_j$ is indeed a monic polynomial in $\ZZ[\zeta,c]$.


\begin{lemma}
\label{lem:b_i}
Let $d,m,n\geq 2$. Let $\ell$ be a proper divisor of $n$, and let $\alpha_0$ be a root of $G_{d,m,\ell}^{\zeta}$,
where $\zeta\neq 1$ is a $d$-th root of unity. Set $L=\QQ(\alpha_0)$. Then
\[ \big\la G_{d,m,n}^{\zeta}(\alpha_0)\big\ra = \prod_{j| (n/\ell)} \big\la b_j(\alpha_0)\big\ra^{\mu(n/ (j \ell))}\]
as ideals in $\calO_L$.
\end{lemma}

\begin{proof}
Suppose first that $n\nmid m-1$. By the definition of $G_{d,m,n}^{\zeta}$, we have
\[G_{d,m,n}^{\zeta} =
\prod_{k|n} \big(a_{m+k-1}-\zeta a_{m-1}\big)^{\mu(n/k)} = H_1 \cdot H_2, \]
where
\[ H_1 := \prod_{\substack{t|\ell \\ t\neq \ell}} \prod_{\substack{k|n\\ \gcd(k,\ell)=t}}
\big(a_{m+k-1}-\zeta a_{m-1}\big)^{\mu(n/k)}
\quad\text{and}\quad
H_2 := \prod_{\substack{k|n \\ \ell | k}}
\big(a_{m+k-1}-\zeta a_{m-1}\big)^{\mu(n/k)}. \]
By Remark~\ref{rem:ndivu}, none of the polynomials $a_{m+k-1}-\zeta a_{m-1}$ appearing
in the product defining $H_1$ has a zero at $\alpha_0$, whereas all of the terms
in the product defining $H_2$ are zero at $\alpha_0$.

For any \emph{proper} divisor $t$ of $\ell$, 
and for any integer $k\geq 1$ with $\gcd(k,\ell)=t$,
Lemma~\ref{lem:samev} and the fact that $f_{d,\alpha_0}$ has exact type $(m,\ell)$ yield
\[\big\la a_{m+k-1}(\alpha_0)-\zeta a_{m-1}(\alpha_0)\big\ra =
\big\la a_{m+t-1}(\alpha_0)-\zeta a_{m-1}(\alpha_0)\big\ra\]
as ideals in $\calO_L$.
It follows that for any proper divisor $t$ of $\ell$, we have
\[ \prod_{\substack{k|n\\ \text{gcd}(k,\ell)=t}}
\big\la a_{m+k-1}(\alpha_0)-\zeta a_{m-1}(\alpha_0)\big\ra^{\mu(n/k)}
= \big\la a_{m+t-1}(\alpha_0)-\zeta a_{m-1}(\alpha_0)\big\ra^{E_t} ,\]
where
\[ E_t := \sum_{\substack{k|n \\ \gcd(k,\ell)=t}} \mu\bigg(\frac{n}{k}\bigg) . \]
However, according to Lemma~\ref{lem:elementary}, we have $E_t=0$.
Taking the product over all such $t$,
it follows that the ideal $\big\la H_1(\alpha_0) \big\ra$ is simply the identity ideal $\calO_L$.

By writing $k=j\ell$, we may rewrite the product defining $H_2$ as
\[ H_2 = \prod_{j|(n/\ell)} B_j^{\mu(n/(j\ell))}
= \prod_{j|(n/\ell)} \bigg(\frac{B_j}{B_1}\bigg)^{\mu(n/(j\ell))}
= \prod_{j|(n/\ell)} b_j^{\mu(n/(j\ell))}, \]
by the definition of $B_j$ and $b_j$ from equation~\eqref{eq:Bdef},
where we have used the fact that $\ell$ is a \emph{proper} divisor of $n$,
and hence that $\sum_{j|(n/\ell)} \mu( n/(j\ell)) = 0$. Thus, we have
\[ \big\la G_{d,m,n}^{\zeta}(\alpha_0)\big\ra = \big\la H_1(\alpha_0) \big\ra \big\la H_2(\alpha_0) \big\ra
= \prod_{j| (n/\ell)} \big\la b_j(\alpha_0)\big\ra^{\mu(n/ (j \ell))}, \]
completing the proof in the case that $n\nmid m-1$.
 
Next, we suppose that $n \, | \, m-1$. With $H_1$ and $H_2$ as before, we have
\[\la G_{d,m,n}^{\zeta}(\alpha_0)\ra = \la H_1(\alpha_0)\ra \cdot \la H_2(\alpha_0)\ra
\cdot  \prod_{k|n} \la a_k(\alpha_0)\ra^{-\mu(n/k)} .\]
Hence, to finish the proof, it suffices to show that
\[ \prod_{k|n} \big\la a_k(\alpha_0)\big\ra^{\mu(n/k)}=\calO_L. \]
Since $\alpha_0$ is a root of $G_{d,m,\ell}^{\zeta}$,
Proposition~3.1 of \cite{BG1} shows that if $\ell | k$, then
$\la a_k(\alpha_0)\ra=\la a_{\ell}(\alpha_0)\ra$ as ideals in $\calO_L$;
and if $\ell\nmid k$, then $a_k(\alpha_0)$ is a unit in $\calO_L$,
so that $\la a_k(\alpha_0)\ra=\calO_L$. Thus,
\[ \prod_{k|n} \big\la a_k(\alpha_0)\big\ra^{\mu(n/k)}
= \prod_{j|(n/\ell)} \big\la a_{\ell}(\alpha_0)\big\ra^{\mu(n/(j\ell))} = \calO_L, \]
where the last equality is
because the nonzero ideal $\la a_{\ell}(\alpha_0)\ra$
is being raised to the power $\sum_{j|(n/\ell)}\mu(\frac{n}{j\ell})=0$,
since $n / \ell > 1$.
\end{proof}

\section{Results on multipliers}
\label{sec:mult}
For a Misiurewicz parameter $\alpha_0$ with type $(m,\ell)$,
the algebraic integers $b_j(\alpha)$ defined by equation~\eqref{eq:Bdef} of Section~\ref{sec:techlem}
turn out to be related to the multiplier of the periodic cycle in the associated critical orbit,
as we now discuss.

Fix integers $d,m\geq 2$ and $n \geq 1$.
Let $\zeta\neq 1$ be a $d$-th root of unity.
Consider a root $\alpha_0$ of $G_{d,m,n}^\zeta$.
Since $f:=f_{d,\alpha_0}$ has $f'(z)=dz^{d-1}$, equation~\eqref{eq:multdef} shows that the multiplier
of the periodic cycle $\{a_m(\alpha_0),\dots, a_{m+n-1}(\alpha_0)\}$ is
\begin{equation}
\label{eq:lambda}
\lambda_{d,m,n}^{\zeta}(\alpha_0) = d^n \Bigg(\prod_{i=0}^{n-1}a_{m+i-1}(\alpha_0)\Bigg)^{d-1}.
\end{equation}

The following result shows that the sequence $\{b_j(\alpha_0)\}_{j\geq 1}$ satisfies a linear recurrence, which will immediately allow us to write a simple and explicit formula for $b_j(\alpha_0)$.

\begin{thm}
\label{thm:recur}
Fix integers $d,m,n\geq 2$, a proper divisor $\ell\geq 1$ of $n$,
and a $d$-th root of unity $\zeta\neq 1$. Define
\[ C_{d,m,\ell}:= d^{\ell} \bigg( \prod_{i=0}^{\ell-1} a_{m-1+i} \bigg)^{d-1} \in \ZZ[c] .\]
Then for every integer $j\geq 1$, we have
\[ b_{j+1} \equiv \zeta^{d-1}C_{d,m,\ell} b_j + 1 \pmod{I}, \]
where $I:=\la B_1 \ra$ is the principal ideal of $\ZZ[\zeta,c]$ generated by $B_1=a_{m+\ell - 1}-\zeta a_{m-1}$.
In particular, for any root $\alpha_0$ of $B_1$, we have
\[ b_{j+1}(\alpha_0) = \zeta^{d-1} C_{d,m,\ell}(\alpha_0) b_j(\alpha_0) + 1. \]
\end{thm}

\begin{proof}
Clearly, we have $a_{m+\ell-1} \equiv \zeta a_{m-1} \pmod{I}$,
and hence, repeatedly raising each side to the power $d$, we have
$a_k\equiv a_{k+\ell} \pmod{I}$ for every $k\geq m$.
Therefore, for every $j\geq 1$ and $i\geq 0$, we have
\begin{equation}
\label{eq:keycong}
a_{m+\ell j-1} \equiv \zeta a_{m-1} \pmod{I}
\quad\text{and}\quad
a_{m+\ell j+i} \equiv a_{m+i} \pmod{I} .
\end{equation}
Also observe that for any $P, Q\in\ZZ[\zeta,c]$, if we let $f(z):=z^d+c$, we have
\[ f(P)-f(Q) = P^d - Q^d = (P-Q) \sum_{t=0}^{d-1} P^t Q^{d-1-t} .\]
Thus, for any $j\geq 1$, applying the first congruence of~\eqref{eq:keycong} gives
\begin{align}
\label{eq:subcong1}
a_{m+\ell j} - a_{m} &=
(a_{m+\ell j-1} - \zeta a_{m-1} ) \sum_{t=0}^{d-1} a_{m+\ell j-1}^t (\zeta a_{m-1})^{d-1-t}
\\
& \equiv d (\zeta a_{m-1})^{d-1} (a_{m+\ell j-1} - \zeta a_{m-1} ) \pmod{I^2} ,
\notag
\end{align}
where we have used the fact that $a_{m+\ell j-1} - \zeta a_{m-1}\in I$ to obtain
the congruence modulo $I^2$.
Similarly, for every $j\geq 1$ and every $i=1,\ldots,\ell-1$,
the second congruence of~\eqref{eq:keycong} yields
\begin{align}
\label{eq:subcong2}
a_{m+\ell j+i} - a_{m+i} &=
(a_{m+\ell j+i-1} - a_{m+i-1} ) \sum_{t=0}^{d-1} a_{m+\ell j+i-1}^t (a_{m+i-1})^{d-1-t}
\\
& \equiv d (a_{m+i-1})^{d-1} (a_{m+\ell j+i-1} - a_{m+i-1} ) \pmod{I^2} .
\notag
\end{align}
Thus, for any $j\geq 1$, working in the ring $\ZZ[\zeta,c]$ modulo $I^2$, we have
\begin{align*}
\zeta^{d-1} C_{d,m,\ell} B_j &=
\Bigg[ d^{\ell-1} \Bigg( \prod_{i=1}^{\ell-1} a_{m-1+i} \Bigg)^{d-1}\Bigg] \cdot
d (\zeta a_{m-1})^{d-1} (a_{m+\ell j-1} - \zeta a_{m-1} )
\\
&\equiv
d^{\ell-1} \Bigg( \prod_{i=1}^{\ell-1} a_{m-1+i} \Bigg)^{d-1} (a_{m+\ell j} - a_m)
\\
&=
\Bigg[ d^{\ell-2} \Bigg( \prod_{i=2}^{\ell-1} a_{m-1+i} \Bigg)^{d-1} \Bigg] \cdot
d (a_{m})^{d-1} (a_{m+\ell j} - a_{m} )
\\
&\equiv
d^{\ell-2} \Bigg( \prod_{i=2}^{\ell-1} a_{m-1+i} \Bigg)^{d-1} (a_{m+\ell j+1} - a_{m+1})
\\
&\cdots
\\
&\equiv a_{m+\ell j+\ell-1} - a_{m+\ell-1} = B_{j+1} - B_1,
\end{align*}
where we have used identity~\eqref{eq:subcong1} in the first congruence,
and identity~\eqref{eq:subcong2} in all the subsequent congruences.
Since these are congruences modulo $I^2$, where $I=\la B_1 \ra$,
and since $B_k\in I$ for every $k$ (by the first congruence of~\eqref{eq:keycong}),
we may divide both sides by $B_1$ to obtain
\[ \zeta^{d-1} C_{d,m,\ell} b_j \equiv b_{j+1} - 1 \pmod{I}
\quad\text{for every } j\geq 1, \]
yielding the first desired conclusion.
Evaluating both sides at $c=\alpha_0$ and using the fact that $F(\alpha_0)=0$
for every $F\in I$, the second conclusion follows immediately.
\end{proof}

\begin{cor}
\label{cor:recur}
With notation as in Theorem~\ref{thm:recur}, for every $j\geq 1$ we have
\[ b_j(\alpha_0) = 1+C+\cdots + C^{j-1} ,\]
where $C=\zeta^{d-1} C_{d,m,\ell}(\alpha_0)$.
\end{cor}

\begin{proof}
We have $b_1(\alpha_0)=1$ by definition, and $b_{j+1}(\alpha_0)=Cb_j(\alpha_0)+1$
for every $j\geq 1$ by Theorem~\ref{thm:recur}.
The conclusion is immediate by induction.
\end{proof}

\begin{remark}
\label{rem:mult}
With the notation in Corollary~\ref{cor:recur}, we have $C=\lambda_{d,m,\ell}^{\zeta}(\alpha_0)$.
Indeed, substituting $a_{m+\ell-1}(\alpha_0) = \zeta a_{m-1}(\alpha_0)$ into
$C=\zeta^{d-1}d^{\ell} \prod_{i=0}^{\ell-1} a_{m-1+i}^{d-1}(\alpha_0)$
immediately yields
$C= d^{\ell} \prod_{i=0}^{\ell-1} a_{m+i}^{d-1}(\alpha_0) = \lambda_{d,m,\ell}^{\zeta}(\alpha_0)$.
\end{remark}

The following result
reduces Conjecture~\ref{conj:j=m} to a question about multipliers
and cyclotomic polynomials $\Phi_j$.

\begin{prop}
\label{prop:reduction}
Fix integers $d,m,n\geq 2$, a proper divisor $\ell\geq 1$ of $n$,
and a $d$-th root of unity $\zeta\neq 1$.
Let $\alpha_0$ be a root of $G_{d,m,\ell}^{\zeta}$, and set $L=\QQ(\alpha_0)$.
Suppose that $G_{d,m,n}^{\zeta}$ is irreducible over $\QQ(\zeta)$,
and let $c_0$ be one of its roots. Set $K=\QQ(c_0)$.
If $\Phi_{n/\ell}(\lambda_{d,m,\ell}^{\zeta}(\alpha_0))$ is not a unit in $\calO_L$,
then $G_{d,m,\ell}^{\zeta}(c_0)$ is not a unit in $\calO_K$.
\end{prop}

\begin{proof}
Define $C=\zeta^{d-1} C_{d,m,\ell}(\alpha_0)$, which is the multiplier
$\lambda_{d,m,\ell}^{\zeta}(\alpha_0)$
of the periodic cycle of $f:=f_{d,\alpha_0}$, by Remark~\ref{rem:mult}.
If $C$ is a root of unity (in particular, if $C=1$),
then $f$ has a parabolic periodic point, and hence by basic complex dynamics
(see, for example, \cite[Theorem~2.3]{CG} or \cite[Corollary~14.5]{Mil}),
some critical point of $f$ must be wandering, contradicting the fact
that the only two critical points of $f$ (at $z=0$ and $z=\infty$) are preperiodic.
Thus, $C\neq 1$, and hence Corollary~\ref{cor:recur} yields
\[ b_j(\alpha_0) = \frac{C^j - 1}{C-1} \quad \text{for every integer } j\geq 1 .\]
Therefore, by Lemma~\ref{lem:b_i}, there is a unit $u\in \calO_L$ such that
\[G_{d,m,n}^{\zeta}(\alpha_0) =
u \prod_{j| (n/\ell)} \bigg(\frac{C^j -1}{C-1}\bigg)^{\mu(n/ (j\ell))} =
u \prod_{j| (n/\ell)} \big(C^j -1\big)^{\mu(n/ (j\ell))},\]
where the second equality is because $C-1\neq 0$
is being raised to the power $\sum_{j|(n/\ell)}\mu(\frac{n}{j\ell})=0$,
Hence, by the definition of the cyclotomic polynomial $\Phi_{n/\ell}$, we obtain
\[G_{d,m,n}^{\zeta}(\alpha_0) = u \Phi_{n/\ell}(C) = u \Phi_{n/\ell}\big(\lambda_{d,m,\ell}^{\zeta}(\alpha_0)\big).\]
The result now immediately follows from Remark~\ref{rem:alpha_0_unit}.
\end{proof}

\section{Multiplier polynomials}
\label{sec:multpoly}

In light of Proposition~\ref{prop:reduction},
we state the following definition and conjecture.

\begin{defin}
\label{def:multpoly}
Let $d,m\geq 2$ and $n\geq 1$ be integers, and let $\zeta\neq 1$ be a $d$-th root of unity.
Let $c_1,\dots,c_k$ be all the roots of $G_{d,m,n}^{\zeta}$. The \emph{multiplier polynomial}
$P_{d,m,n}^{\zeta}$ associated with $G_{d,m,n}^{\zeta}$ is
\[P_{d,m,n}^{\zeta}(x)= \prod_{j=1}^{k} (x-\lambda_{d,m,n}^{\zeta}(c_j)) \in \ZZ[\zeta][x], \]
where $\lambda_{d,m,n}^\zeta$ is defined as in equation~\eqref{eq:lambda}.
\end{defin}

\begin{conj}
\label{conj:resultant}
Let $d,m\geq 2$ and $n\geq 1$ be integers. Let $c_1,\dots,c_k$ be roots of $G_{d,m,n}^{\zeta}$, where $\zeta\neq 1$ is a $d$-th root of unity. Then for every integer $i\geq 1$, the resultant
$\Res(P_{d,m,n}^{\zeta},\Phi_{i})$ is not a unit in $\ZZ[\zeta]$,
where $\Phi_i$ is the $i$-th cyclotomic polynomial.
\end{conj}

The following result shows that if every $G_{d,m,n}^{\zeta}$ is irreducible over $\QQ(\zeta)$,
then Conjecture~\ref{conj:resultant} implies Conjecture~\ref{conj:j=m}.

\begin{prop}
\label{prop:implies}
Let $d,m,n\geq 2$ be integers. Let $c_0$ be a root of $G_{d,m,n}^{\zeta}$,
where $\zeta\neq 1$ is a $d$-th root of unity, and set $K=\QQ(c_0)$.
Assume that $G_{d,m,t}^{\zeta}$ is irreducible over $\QQ(\zeta)$ for any divisor $t$ of $n$.
Let $\ell\geq 1$ be a proper divisor of $n$, and suppose that $\Res(P_{d,m,\ell}^{\zeta}, \Phi_{n/\ell})$
is not a unit in $\ZZ[\zeta]$. Then $G_{d,m,\ell}^{\zeta}(c_0)$ is not a unit in $\calO_K$.
\end{prop}

\begin{proof}
Fix a proper divisor $\ell$ of $n$.
Enumerate the roots of $G_{d,m,\ell}^{\zeta}$ as $\alpha_1,\dots,\alpha_k$.
Let $r\geq 1$ be the number of indices $j\in \{1,\dots,k\}$ for which
$\lambda_{d,m,\ell}^{\zeta}(\alpha_1) = \lambda_{d,m,\ell}^{\zeta}(\alpha_j)$.
Because $\lambda_{d,m,\ell}^{\zeta}\in \ZZ[\zeta][c]$ is a polynomial over $\QQ(\zeta)$,
and because we have assumed $G_{d,m,\ell}^{\zeta}$ is irreducible over $\QQ(\zeta)$,
we know that
\[ \lambda_{d,m,\ell}^{\zeta}(\alpha_1),\dots,\lambda_{d,m,\ell}^{\zeta}(\alpha_k) \]
are Galois conjugates over $\QQ(\zeta)$.
It follows that for any $1\leq i\leq k$, there are exactly $r$ distinct indices $j\in \{1,\dots,k\}$ such that
$\lambda_{d,m,\ell}^{\zeta}(\alpha_i) = \lambda_{d,m,\ell}^{\zeta}(\alpha_j)$.
Therefore, the minimal polynomial $h_{d,m,\ell}^{\zeta}$
of $\lambda_{d,m,\ell}^{\zeta}(\alpha_i)$ over $\QQ(\zeta)$ for $i=1,\dots,k$ satisfies
\begin{equation}
\label{eq:powerP}
(h_{d,m,\ell}^{\zeta})^r = P_{d,m,\ell}^{\zeta}.
\end{equation}

Set $L_i=\QQ(\alpha_i)$ for $1\leq i\leq k$. By the definition of resultant, we have
\[ \Res\big(h_{d,m,\ell}^{\zeta}, \Phi_{n/\ell}\big) =
\text{N}_{L_i/\QQ(\zeta)}\Big(\Phi_{n/\ell}\big(\lambda_{d,m,\ell}^{\zeta}(\alpha_i)\big)\Big)
\quad\text{for } i=1,\ldots, k, \]
where $N_{L_i/\QQ(\zeta)}$ denotes the norm from $L_i$ to $\QQ(\zeta)$.
Therefore, if $\Res(h_{d,m,\ell}^{\zeta}, \Phi_{n/\ell})$ is not a unit in $\ZZ[\zeta]$,
then $G_{d,m,\ell}^{\zeta}(c_0)$ is not a unit in $\calO_K$ by Proposition~\ref{prop:reduction}.
(Recall that we have assumed $G_{d,m,n}^{\zeta}$ is irreducible over $\QQ(\zeta)$.)
By equation~\eqref{eq:powerP} and the multiplicativity of the resultant, we have
\[\Res(h_{d,m,\ell}^{\zeta}, \Phi_{n/\ell})\text{ is not a unit in }\ZZ[\zeta] \iff
\Res(P_{d,m,\ell}^{\zeta}, \Phi_{n/\ell})\text{ is not a unit in }\ZZ[\zeta]\]
and hence the result follows from our hypothesis
that $\Res(P_{d,m,\ell}^{\zeta}, \Phi_{n/\ell})$ is not a unit.
\end{proof}

\section{$p$-special polynomials}
\label{sec:special}
In order to study the resultants $\Res(P_{d,m,\ell}^{\zeta}, \Phi_{n/\ell})$ arising
in Section~\ref{sec:multpoly}, we wish to show that the multiplier polynomials $P_{d,m,\ell}^{\zeta}$
belong to a special class, which we now define.

\begin{defin}
\label{def:special}
Let $P(x)=x^i+A_{i-1}x^{i-1}+\dots+A_1x+A_0\in\ZZ[x]$ be a monic polynomial with integer coefficients,
and let $p$ be a prime number.
We say that $P(x)$ is \emph{p-special} if it satisfies the following two properties:
\begin{itemize}
	\item $v_p(A_{i-1})> v_p(2)$, and
	\item $v_p(A_j)>v_p(A_{i-1})$ for $j=0,1,\dots,i-2$.
\end{itemize}
\end{defin}

Note that if $p=2$, the first bullet point in Definition~\ref{def:special} says that $v_p(A_{i-1})\geq 2$,
whereas for $p\geq 3$, it says that $v_p(A_{i-1})\geq 1$.
Our main result about $p$-special polynomials is the following.

\begin{thm}
\label{thm:special}
Let $P(x)\in \ZZ[x]$ be a $p$-special polynomial for some prime $p$.
Then for every integer $\ell \geq 1$, we have $|\Res(P,\Phi_{\ell})|>1$.
\end{thm}

To prove the theorem, we will need a lemma.

\begin{lemma}
\label{lem:power-special}
Let $P(x)\in \ZZ[x]$ be a $p$-special polynomial for some prime $p$.
For any $n\geq 1$, $(P(x))^n$ is also a $p$-special polynomial.
\end{lemma}

\begin{proof}
We first claim that $(P(x))^p$ is a $p$-special polynomial.
Indeed, write
\begin{equation}
\label{eq:Pexpand}
P(x)=x^i+ \big(A_{i-1} x^{i-1} + \cdots + A_0\big).
\end{equation}
Expanding $(P(x))^p$ using the binomial theorem,
the leading term is $x^{ip}$, and the second term is $pA_{i-1}x^{pi-1}$,
which satisfies $v_p(pA_{i-1}) > v_p(A_{i-1}) > v_p(2)$,
verifying the first bullet point of Definition~\ref{def:special}.
Moreover, since $p \, | \, \binom{p}{j}$ for $1\leq j\leq p-1$,
every other coefficient must be a sum of integers, each of which is divisible
either by $pA_j$ for some $j\in\{0,1\dots,i-2\}$, or by $A_j^p$ for some $j\in\{0,1,\dots,i-1\}$.
For $0\leq j\leq i-2$, we have
\[ v_p(pA_j) > v_p(pA_{i-1}), \]
and for $0\leq j\leq i-1$, we have
\[ v_p(A_j^p) =pv_p(A_j) > 1 + v_p(A_j) = v_p(pA_{j}) \geq v_p(pA_{i-1}), \]
using the fact that $P$ is $p$-special for both inequalities.
These bounds verify the second bullet point for $P^p$, proving the claim.

Second, we claim that for any integer $m\geq 1$ with $p\nmid m$, the polynomial $P^m$
is also $p$-special. Again writing $P$ as in equation~\eqref{eq:Pexpand} and expanding $P^m$,
the lead term is $x^{im}$, and the second term is $mA_{i-1} x^{im-1}$.
We have $v_p(mA_{i-1}) = v_p(A_{i-1}) > v_p(2)$,
verifying the first bullet point of Definition~\ref{def:special}.
All other coefficients are sums of integers divisible either by $A_j A_k$
for some not necessarily distinct $j,k\in\{0,1,\dots,i-1\}$, or else divisible by $A_j$
for some $j\in\{0,1,\ldots,i-2\}$.
For $j,k\in\{0,1,\dots,i-1\}$, we have
\[ v_p(A_jA_k) > v_p(A_j) > v_p(A_{i-1}) =  v_p(mA_{i-1}), \]
and for $0\leq j\leq i-2$, we have
\[ v_p(A_j) > v_p(A_{i-1}) = v_p(mA_{i-1}), \]
using the facts that $v_p(m)=0$ and that $P$ is $p$-special.
These bounds verify the second bullet point for $P^m$, proving our second claim.

Inductively applying the two above claims, it follows that $P^n$ is $p$-special
for any positive integer $n$.
\end{proof}

\begin{proof}[Proof of Theorem~\ref{thm:special}]
It suffices to show that for any primitive $\ell$-th root of unity $\zeta$, the value
$P(\zeta)$ is not a unit in the ring $\ZZ[\zeta]$. For the sake of contradiction, we suppose that
\begin{equation}
\label{eq:20}
P(\zeta)\in \ZZ[\zeta]^{\times}.
\end{equation}

\textbf{Case 1}.
Assume $\ell=1.$ Write $P(x)=x^i+\sum\limits_{t=0}^{i-1} A_tx^t$. Since $\zeta =1$,
assumption~\eqref{eq:20} yields
\begin{equation}
\label{eq:ell1}
1+\sum_{t=0}^{i-1} A_t=\pm 1.
\end{equation}
If the right-hand side of equation~\eqref{eq:ell1} is 1, we have
\begin{equation}
\label{eq:ell1c1}
A_{i-1} = -\sum_{t=0}^{i-2}A_t,
\end{equation}
which is impossible because all the terms on the right have valuation strictly greater than the
term on the left.
Similarly, if the right-hand side of equation~\eqref{eq:ell1} is $-1$, we have
\begin{equation}
\label{eq:ell1c2}
-2 = \sum_{t=0}^{i-1}A_t,
\end{equation}
which again is impossible,
because $v_p(A_t) > v_p(2)=v_p(-2)$ for all $0\leq t\leq i-1$.
Either way, we have our desired contradiction.

\smallskip

\textbf{Case 2}.
Assume $\ell=2$, so that $\zeta=-1$. With notation as in Case~1, we have 
\[ (-1)^i+\sum_{t=0}^{i-1} (-1)^{t}A_{t}=\pm 1,
\quad\text{i.e.,}\quad \sum_{t=0}^{i-1} (-1)^{t}A_{t} \in \{-2,0,2\} . \]
If the sum above is zero, we reach a contradiction in the same way as in equation~\eqref{eq:ell1c1};
or if it is $\pm 2$, we reach a contradiction in the same way as in equation~\eqref{eq:ell1c2}.

\smallskip

\textbf{Case 3}.
Assume for the rest of the proof that $\ell\geq 3$.
Note that the subgroup $\langle \zeta\rangle \ZZ[\zeta+\zeta^{-1}]^{\times}$
of the unit group $\ZZ[\zeta]^{\times}$
has index at most two inside of $\ZZ[\zeta]^{\times}$,
by \cite[Theorem 4.12]{Washington96}.
Define $Q(x)\in\ZZ[x]$ by $Q:=P$ if this index is $1$, or $Q:=P^2$ if this index is $2$.
Then $Q$ is $p$-special by Lemma~\ref{lem:power-special}, and
$Q(\zeta)\in \langle \zeta\rangle \ZZ[\zeta+\zeta^{-1}]^{\times}$.
Assumption~\eqref{eq:20} therefore implies that
\begin{equation}
\label{eq:21}
Q(\zeta) = \zeta^iF(\zeta+\zeta^{-1})
\end{equation}
for some polynomial $F(x)\in \ZZ[x]$ and some integer $i$.
Setting $R:=Q^{\ell}$, which has degree $k\ell$ for some positive integer $k$,
we have that $R$ is again $p$-special by Lemma~\ref{lem:power-special}.
Taking the $\ell$-th power of both sides of equation~\eqref{eq:21} yields
\begin{equation}
\label{eq:22}
R(\zeta) = G(\zeta+\zeta^{-1})
\end{equation}
for some polynomial $G(x)\in \ZZ[x]$ with
$\deg(G)<\phi(\ell)/2$, where $\phi$ is the Euler totient function,
since $[\QQ(\zeta+\zeta^{-1}):\QQ]=\phi(\ell)/2$.

Write
\[ R(x)=x^{k\ell}+\sum_{t=0}^{k\ell-1}B_{t}x^{t}
\quad\text{and}\quad G(x)=\sum_{t=0}^{m}D_tx^t .\]
Substituting in~\eqref{eq:22}, we obtain
\[1+\sum_{t=0}^{k\ell-1}B_{t}\zeta^t=\sum\limits_{t=0}^{m} D_t(\zeta+\zeta^{-1})^t.\]
Multiplying by $\zeta^m$ and moving one term to the other side yields
\begin{equation}
\label{eq:23}
\sum_{t=0}^{k\ell-1} B_{t}\zeta^{m+t}=
\Bigg(\sum_{t=1}^{m} D_t(\zeta^2+1)^t\zeta^{m-t}\Bigg)+(D_0-1)\zeta^m.
\end{equation}
Let $r:=v_p(B_{k\ell-1}) > v_p(2)$.
Since $R$ is a $p$-special polynomial, the left-hand side in \eqref{eq:23} lies in $p^r\ZZ[\zeta]$.
We claim that $D_t\in p^r\ZZ$ for each $t=1,2,\dots,m$.

To prove the claim, suppose it were false, and consider the largest index $i\in\{1,2,\dots,m\}$ for which
$D_i\not\equiv 0 \pmod{p^r}$.
Note that the term with the largest power of $\zeta$ in $D_i(\zeta^2+1)^i\zeta^{m-i}$
is $D_i\zeta^{m+i}$, and for any $t<i$ with $D_t\not \equiv 0 \pmod{p^r}$,
the term $D_t(\zeta^2+1)^t\zeta^{m-t}$ only involves powers $\zeta^e$ with $0\leq e <m+i$.
Hence, recalling that $0\leq m<\phi(\ell)/2$, and thus $m+i<\phi(\ell)$,
it follows immediately that the right-hand side of equation~\eqref{eq:23},
when written as a $\ZZ$-linear combination of the integral basis $\{1,\zeta,\zeta^2,\ldots,\zeta^{\phi(\ell)-1}\}$,
has $\zeta^{m+i}$-coefficient not congruent to $0$ modulo $p^r$.
Thus, the right-hand side of equation~\eqref{eq:23} cannot lie in
$p^r\ZZ[\zeta]$, which is a contradiction, proving our claim.

By the claim, it follows that both the left-hand side and the sum in parentheses in equation~\eqref{eq:23}
must lie in $p^r\ZZ[\zeta]$. Therefore, we must also have $(D_0-1)\zeta^m\in p^r \ZZ[\zeta]$.
Define $B_t' : =p^{-r} B_t\in\ZZ$ for each $0\leq t\leq k\ell -1$,
and $D_t' : =p^{-r} D_t\in\ZZ$ for each $1\leq t\leq m$,
and $D_0' : =p^{-r} (D_0-1)\in\ZZ$.
Then equation~\eqref{eq:23} becomes
\begin{equation}
\label{eq:prgone}
\sum_{t=0}^{k\ell-1}B'_{t}\zeta^{m+t}=\sum_{t=0}^{m}D'_t(\zeta^2+1)^t\zeta^{m-t}.
\end{equation}
Note that $B'_{k\ell-1}\not\equiv 0 \pmod{p}$, and $B'_t\equiv 0\pmod{p}$ for $t=0,\dots k\ell-2$,
because $R$ is $p$-special. In addition, we have $\zeta^{m+k\ell-1}=\zeta^{m-1}$.
Thus, reducing equation~\eqref{eq:prgone} modulo $p\ZZ[\zeta]$ yields
\begin{equation}
\label{eq:redprgone}
\Bigg(\sum_{t=0}^{m}D'_t(\zeta^2+1)^t\zeta^{m-t}\Bigg)-B'_{k\ell-1}\zeta^{m-1}\equiv 0\pmod{p\ZZ[\zeta]}.
\end{equation}
Because we assumed $\ell\geq 3$, we have $\zeta^m\neq \pm \zeta^{m-1}$.
If $D'_1,D'_2,\ldots,D'_m$ were all divisible by $p$, we would obtain
$D'_0\zeta^m-B_{k\ell-1}'\zeta^{m-1}\equiv 0 \pmod{p\ZZ[\zeta]}$, which is impossible.
Thus, we may define $i$ to be the largest index $1\leq i\leq m$ for which $p\nmid D'_i$.

Note that the term with the largest power of $\zeta$ in $D_i'(\zeta^2+1)^{i}\zeta^{m-i}$ is $D_i'\zeta^{m+i}$.
As above, for any positive integer $t<i$ with $p\nmid D'_i$,
the term $D'_t(\zeta^2+1)^t\zeta^{m-t}$ only involves powers $\zeta^a$ with $a<m+i$.
Hence, when the left side of equation~\eqref{eq:redprgone} is fully expanded,
the coefficient of $\zeta^{m+i}$ is relatively prime to $p$.
This contradicts equation~\eqref{eq:redprgone} itself, completing the proof of Theorem~\ref{thm:special}.
\end{proof}

\section{The case $d=2$: Theorem~\ref{thm:d=n=2} via 2-special polynomials}
\label{sec:imply}

For the remainder of the paper, we consider Conjecture~\ref{conj:resultant} in the case $d=2$,
i.e., $f_{d,c}(z)=z^2+c$.
For ease of notation, we will hereafter write simply
$G_{m,n}$, $P_{m,n}$, and $\lambda_{m,n}$ instead of
$G_{2,m,n}^{\zeta}$, $P_{d,m,n}^{\zeta}$, and $\lambda_{d,m,n}^{\zeta}$, respectively.

With this notation, we propose the following conjecture, which we have verified using Magma
for all pairs $(m,n)$ with $m+n\leq 10$. By Theorem~\ref{thm:special},
it clearly implies Conjecture~\ref{conj:resultant} in the case $d=2$.


\begin{conj}
\label{conj:special}
Let $m\geq 2$ and $n\geq 1$. Then $P_{m,n}$ is a $2$-special polynomial.
\end{conj}

Most of the rest of this paper is devoted to the proof of the following theorem,
which confirms Conjecture~\ref{conj:special} for $n=1$ and $n=2$.

\begin{thm}
\label{thm:special12}
Let $m\geq 2$ be an integer. Then both $P_{m,1}$ and $P_{m,2}$ are $2$-special.
\end{thm}

Assuming Theorem~\ref{thm:special12}, we can now prove Theorem~\ref{thm:d=n=2}.

\begin{proof}[Proof of Theorem~\ref{thm:d=n=2}]
By \cite[Corollary 1.1]{Gok19}, both $G_{m,1}$ and $G_{m,2}$ are irreducible over $\QQ$.
By Theorem~\ref{thm:special12}, both $P_{m,1}$ and $P_{m,2}$ are $2$-special,
and therefore, by Theorem~\ref{thm:special}, we have
\begin{equation}
\label{eq:Resnotunit}
\Res(P_{m,1},\Phi_i) \text{ and }
\Res(P_{m,2},\Phi_i) \text{ are not units in } \ZZ, \quad\text{for all } i\geq 1.
\end{equation}
With $i=n$ in equation~\eqref{eq:Resnotunit}, then choosing
$\ell=1$ in Proposition~\ref{prop:implies} yields that $G_{m,1}(c_0)$ is not a unit in $\calO_K$.
Similarly, if $n$ is even, choosing $i=n/2$ and $\ell=2$ yields that $G_{m,2}(c_0)$ is not a unit in $\calO_K$.
\end{proof}

It remains only to prove Theorem~\ref{thm:special12}, for which we will need the following
explicit formulas for $P_{m,1}$ and $P_{m,2}$.

\begin{lemma}
\label{lem:formula_P_{m,n}}
Let $m\geq 2$. Then:
\begin{enumerate}[label=\textup{(\alph*)}]
\item $\dsps P_{2,1}(x) = x-4$,
\item $\dsps P_{m,1}(x) = 2^{2^{m-1}} x^{-1}a_{m-1}\bigg(\frac{-x^2+2x}{4}\bigg) +2^{(2^{m-1}-1)}$
if $m\geq 3$, and
\item $\dsps P_{m,2}(x)= 4^k G_{m,2}\bigg(\frac{x-4}{4}\bigg)$,
where $k:=\deg(G_{m,2})$. Then

\end{enumerate}
\end{lemma}

\begin{proof}
\textbf{(a) and (b)}.
Enumerate the roots of $G_{m,1}$ as $c_1,c_2,\dots,c_k$, where $k:=\deg(G_{m,1})=2^{m-1}-1$.
For each $i=1,\dots,k$, define
\[\gamma_i = -a_{m-1}(c_i) = a_m(c_i),\]
so that the multiplier of the cycle in the periodic orbit is $\lambda_i:=2\gamma_i$. Observe that
\[\gamma_i^2+c_i = f_{c_i}(-a_{m-1}(c_i))=a_m(c_i)=\gamma_i,\]
and hence $c_i = -\gamma_i^2+\gamma_i$. Thus,
\[a_{m-1}(-\gamma_i^2+\gamma_i)=a_{m-1}(c_i)=-\gamma_i.\]
Adding $\gamma_i$ to both sides and dividing by $\gamma_i$
--- while remembering that $a_{m-1}$ has no constant term --- we obtain
\[ Q_1(\gamma_i)=0, \quad\text{where}\quad
Q_1(x):=\frac{1}{x}a_{m-1}(-x^2+x)+1.\]
If $m=2$, then $a_{m-1}(x)=x$, and hence $Q_1(x)=2-x$.
On the other hand, if $m\geq 3$, then
$a_{m-1}$ is monic of (even) degree $2^{m-2}$, so that $Q_1$ is also monic, but of degree $2^{m-1}-1$.
Thus, define $\tilde{Q}_1:=-Q_1$ if $m=2$, or $\tilde{Q}_1=Q_1$ if $m\geq 3$.
Then $\tilde{Q}_1$ is monic of degree $2^{m-1}-1$,
which is the same as the degree $k$ of $G_{m,1}$.
Because each root $c_i$ of $G_{m,1}$ yields a root of $\tilde{Q}_1$ via $\gamma_i = -a_{m-1}(c_i)$,
and conversely via $c_i = -\gamma_i^2+\gamma_i$, the roots of $\tilde{Q}_1$
are precisely $\gamma_1,\ldots,\gamma_k$. Therefore, since $\tilde{Q}_1$ is monic, we have
\[ \tilde{Q}_1(x) = \prod_{i=1}^k (x-\gamma_i) .\]

The corresponding multipliers satisfy $\lambda_i = 2\gamma_i$, and hence
\[ P_{m,1}(x) = \prod_{i=1}^k (x-\lambda_i) = 2^{k}\prod_{i=1}^k \bigg(\frac{x}{2}-\gamma_i\bigg)
= 2^{(2^{m-1}-1)}\tilde{Q}_1\bigg(\frac{x}{2}\bigg).\]
If $m=2$, and hence $\tilde{Q}_1(x)=x-2$, this expression is $2(\frac{x}{2}-2) = x-4$,
proving part~(a). Otherwise, it is
\[P_{m,1}(x) = 2^{(2^{m-1}-1)} \bigg[ \frac{2}{x} a_{m-1} \bigg( \frac{-x^2}{4} + \frac{x}{2} \bigg) + 1 \bigg]
= 2^{2^{m-1}} x^{-1} a_{m-1} \bigg( \frac{-x^2+2x}{4} \bigg) + 2^{(2^{m-1}-1)},\]
proving part~(b).

\textbf{Part (c)}.
Enumerate the roots of $G_{m,2}$ as $c_1,\dots,c_k$.
(This time, $k=2^{m-1}$ if $m$ is even, or $k=2^{m-1}-1$ if $m$ is odd, although
we will not need those exact values in this part of the proof.)
For each $i=1,\dots,k$, define
\[\gamma_i:= -a_{m-1}(c_i)a_m(c_i) = a_m(c_i)a_{m+1}(c_i),\]
so that the multiplier of the cycle in the periodic orbit is $\lambda_i = 4\gamma_i$.
Since $a_m(c_i)$ and $a_{m+1}(c_i)$ are both $2$-periodic points of $f_{c_i}$,
they are the two roots of the polynomial $z^2+z+(c_i+1)$,
and therefore their product $\gamma_i$ is the constant term of this polynomial.
That is, $\gamma_i = c_i + 1$. Thus, defining $Q_2$
to be the monic polynomial $Q_2(x):= G_{m,2}(x-1)\in\ZZ[x]$, we have
\[Q_2(x) = \prod_{i=1}^k (x-\gamma_i) .\]

The corresponding multipliers satisfy $\lambda_i = 4\gamma_i$, and hence
\[ P_{m,2}(x) = \prod_{i=1}^k (x-\lambda_i) = 4^{k}\prod_{i=1}^k \bigg(\frac{x}{4}-\gamma_i\bigg)
= 4^k Q_2\bigg(\frac{x}{4}\bigg) = 4^k G_{m,2}\bigg(\frac{x-4}{4}\bigg) . \qedhere \]
\end{proof}

To prove Theorem~\ref{thm:special12},
we will also need the following elementary lemma.

\begin{lemma}
\label{lem:vjbound}
Let $\ell\geq 3$ and $0\leq j\leq 2^{\ell}-3$ be integers. Then
\[ 2^{\ell}  > j + v_2(j) + 1. \]
\end{lemma}

\begin{proof}
Suppose first that $j=2^e$ for some $0\leq e < \ell$. If $e\geq 2$, then $e+1<2^e$, and hence
\[ j + v_2(j) + 1 = 2^e + e + 1 < 2^{e+1} \leq 2^{\ell},\]
as desired. Otherwise, we have $e\leq 1$, and hence $2^e+e+1\leq 4 < 2^{\ell}$, since $\ell\geq 3$.
Either way, we are done.

The only other possibility is that $2^{e-1}<j<2^e$ for some $2\leq e\leq \ell$.
Note that we must have $v_2(j)\leq e-2$ in this case.
If $e\leq \ell-1$, then since $e<2^e$, we have
\[ j + v_2(j) + 1 < 2^e + (e-2) + 1 <2^{e+1} - 1 < 2^{\ell} .\]
Otherwise, we have $e=\ell$. If $v_2(j)\leq 1$, then
because $j\leq 2^{\ell}-3$, we have
\[ j + v_2(j) + 1 \leq 2^{\ell} -3 + 1 + 1 < 2^{\ell}. \]
The only remaining case is that $e=\ell$ and $v_2(j)\geq 2$.
We must have $j\leq 2^{\ell}-2^{r}$, where $r:=v_2(j)\geq 2$, and hence
\[ j + v_2(j) + 1 \leq 2^{\ell} -2^r + r + 1 < 2^{\ell}. \qedhere \]
\end{proof}

\section{The case $d=2$: Proving Theorem~\ref{thm:special12}}
\label{sec:finish}

Throughout this section, write
\begin{equation}
\label{eq:Alidef}
a_{\ell}(c) = \sum_{i=0}^{2^{\ell-1}} A_{\ell,i} c^i \qquad\text{for any }\ell \geq 1.
\end{equation}
Using the fact that $a_{\ell+1}=(a_{\ell})^2 +c$,
a simple induction on $\ell$ shows that for all $\ell\geq 2$, we have
\begin{equation}
\label{eq:Ainduct}
A_{\ell,2^{\ell-1}}=1, \quad
A_{\ell,2^{\ell-1}-1} = 2^{\ell -2}, \quad
A_{\ell,1}=1, \quad\text{and}\quad A_{\ell,0}=0.
\end{equation}
We also have the following technical but more general bound.

\begin{lemma}
\label{lem:valibound}
Let $\ell\geq 2$ and $0\leq i\leq 2^{\ell-1}-2$ be integers. Then
\begin{equation}
\label{eq:valibound}
v_2(A_{\ell,i}) > 2i + \ell - 2^\ell .
\end{equation}
\end{lemma}
\begin{remark}
\label{rem:bound_exception}
For $i=2^{\ell-1}-1$, it follows immediately from (\ref{eq:valibound}) that
$v_2(A_{\ell,i}) = 2i+\ell-2^{\ell}$.
We will need this fact in the proof of Lemma~\ref{lem:valibound}.
\end{remark}
\begin{proof}[Proof of Lemma~\ref{lem:valibound}]
We proceed by induction on $\ell\geq 2$.
For $\ell=2$, we must have $i=0$, and hence the right side of inequality~\eqref{eq:valibound} is
$0+2-4=-2<0$. Therefore, the desired bound holds trivially.


For the rest of the proof, fix some $\ell\geq 2$,
and assume inequality~\eqref{eq:valibound} holds
for that particular $\ell$; we must show it holds for $\ell+1$ as well. We have
\begin{align*}
a_{\ell+1}(c) &= a_{\ell}(c)^2+c =\Bigg(\sum_{i=1}^{2^{\ell-1}} A_{\ell,i}c^i\Bigg)^2+c \\
&= \sum_{i=1}^{2^{\ell-1}} A_{\ell,i}^2 c^{2i}
+2\Bigg(\sum_{1\leq i<j\leq 2^{\ell-1}} A_{\ell,i}A_{\ell,j}c^{i+j}\Bigg)+c.
\end{align*}
Therefore, the coefficients $A_{\ell+1,i}$ of $a_{\ell+1}$ are given by
\begin{equation}
\label{eq:Amicases}
A_{\ell+1,i} = \begin{cases}
\dsps A_{\ell,i/2}^2+\sum_{\substack{ 0\leq r< s\leq 2^{\ell-1}\\ r+s=i}} 2A_{\ell,r}A_{\ell,s}
& \text{ if $i$ is even}, \\
\dsps \sum_{\substack{ 0\leq r< s\leq 2^{\ell-1}\\ r+s=i}} 2A_{\ell,r}A_{\ell,s}
& \text{ if $i>1$ is odd}, \\
1 & \text{ if } i=1.
\end{cases}
\end{equation}
Since $A_{\ell+1,0}=0$, to finish the induction, we must show that
\begin{equation}
\label{eq:simpbd}
v_2(A_{\ell+1,i}) \overset{\text{?}}{>} \ell+1+2i-2^{\ell+1}
\end{equation}
for all $1\leq i\leq 2^{\ell}-2$.

Given such $i$, suppose first that $i$ is even. We have
\[ v_2\big(A^2_{\ell,i/2}\big) = 2v_2\big(A_{\ell,i/2}\big)
\geq 2(\ell+i-2^{\ell}) > \ell+1 + 2i - 2^{\ell+1}, \]
where the first inequality is by Remark~\ref{rem:bound_exception} (needed for the case $i=2^{\ell}-2$) and our inductive hypothesis that \eqref{eq:valibound} holds for $\ell$,
and the second is because $\ell\geq 2$, and hence $2\ell > \ell+1$. Similarly,
\begin{align}
\label{eq:rsterm}
v_2\big( 2A_{\ell,r}A_{\ell,s} \big)
&= 1+v_2(A_{\ell,r})+v_2(A_{\ell,s})
\geq 1+ v_2(A_{\ell,r}) \notag \\
&> 1+\ell+2r-2^{\ell} \geq \ell+1+2i-2^{\ell+1}
\end{align}
where the first inequality is again by our inductive hypothesis,
and the rest is because $r+s=i$ and $s\leq 2^{\ell-1}$.
Thus, according to equation~\eqref{eq:Amicases}, the desired bound~\eqref{eq:simpbd}
holds when $i$ is even.
 
If $i>1$ is odd, inequality~\eqref{eq:simpbd} is immediate from inequality~\eqref{eq:rsterm}
and equation~\eqref{eq:Amicases}.

Finally, if $i=1$, inequality~\eqref{eq:simpbd} is $0>\ell+3-2^{\ell+1}$,
which holds because $\ell\geq 2$.
\end{proof}

At last, we are prepared to prove Theorem~\ref{thm:special12}.
We treat the cases $n=1$ and $n=2$ separately.

\begin{proof}[Proof of Theorem~\ref{thm:special12} for $n=1$]
For $m=2,3,4$, direct computation with Magma gives:
\[ P_{2,1} = x-4,\quad P_{3,1} = x^3 - 4x^2 + 16, \quad
P_{4,1} = x^7 - 8x^6 + 16x^5 + 16x^4 - 64x^3 + 256, \]
all of which are $2$-special. Thus, we may assume hereafter that $m\geq 5$.

Part (b) of Lemma~\ref{lem:formula_P_{m,n}} therefore gives us
\begin{align*}
P_{m,1}(x) &=
2^{(2^{m-1}-1)} + 2^{2^{m-1}} \sum_{i=1}^{2^{m-2}} A_{m-1,i} x^{-1} \bigg( \frac{2x -x^2}{4}\bigg)^i
\\
&=2^{(2^{m-1}-1)} + \sum_{i=1}^{2^{m-2}} 2^{(2^{m-1}-2i)} A_{m-1,i} x^{i-1} (2-x)^i
\\
&= 2^{(2^{m-1}-1)} + \sum_{i=1}^{2^{m-2}}\sum_{j=0}^{i}2^{(2^{m-1}-i-j)}A_{m-1,i }(-1)^j \binom{i}{j} x^{i+j-1}
\\
&= x^{(2^{m-1}-1)} + \sum_{t=0}^{2^{m-1}-2} B_t x^t,
\end{align*}
where $B_0= 2^{2^{m-1}}$ (using the fact that $A_{m-1,1}=1$), and
\begin{equation}
\label{eq:b_t_formula}
B_t = \mathlarger{\sum_{\substack{ 0\leq j\leq i\leq 2^{m-2}\\ i+j=t+1}}}
2^{2^{m-1}-i-j}A_{m-1,i}(-1)^{j} \binom{i}{j}
\quad\text{for } 1\leq t\leq 2^{m-1}-2 .
\end{equation}
Formula~\eqref{eq:b_t_formula} yields that $B_{2^{m-1}-2}=-2^{m-1}$.
Recalling that $m\geq 5$, it follows that
\[ v(B_{2^{m-1}-2})=m-1\geq 2, \quad\text{and}\quad
v_2(B_0)=2^{m-1}>v_2(B_{2^{m-1}-2}), \]
verifying the first bullet point of Definition~\ref{def:special}, and a small part of the second.
Therefore, to finish verifying that $P_{m,1}$ is $2$-special,
it remains to show that $v_2(B_t)>m-1$ for each integer $1\leq t\leq 2^{m-1}-3$.
We achieve this goal in the following four steps.

\smallskip

\textbf{Step 1 }. We claim that 
\[v_2(B_{2^{m-1}-3})>m-1.\]
Observe that there are only
two pairs of indices $j\leq i\leq 2^{m-2}$ such that $i+j=2^{m-1}-2$, namely
$i=j=2^{m-2}-1$, and $i=2^{m-2}$, $j=2^{m-2}-2$. Therefore, formula~\eqref{eq:b_t_formula} gives
\begin{align*}
B_{2^{m-1}-3} &= -4A_{m-1,2^{m-2}-1}+4A_{m-1,2^{m-2}} \binom{2^{m-2}}{2} \\
& = -4(2^{m-3}) + 4(1) \big(2^{m-3}(2^{m-2}-1)\big) = 2^{m-1}(2^{m-2}-2),
\end{align*}
where we used equation~\eqref{eq:Ainduct} with $\ell=m-1$ in the second equality.
Thus, $v_2(B_{2^{m-1}-3})=m>m-1$, as desired.

\smallskip

\textbf{Step 2}.
We claim that for $i=2^{m-2}$ and $0\leq j\leq 2^{m-2}-3$, we have
\begin{equation}
\label{eq:v2term}
v_2\bigg( 2^{2^{m-1}-i-j}A_{m-1,i}(-1)^{j} \binom{i}{j} \bigg) > m-1.
\end{equation}
By equation~\eqref{eq:Ainduct}, we have $A_{m-1,i}=1$.
Therefore, by the identity $v_2(\binom{2^{m-2}}{j})=m-2-v_2(j)$, 
inequality~\eqref{eq:v2term} becomes
\[ \big(2^{m-1} - 2^{m-2} - j\big) + \big( m-2-v_2(j) \big) > m-1, \]
which holds by applying Lemma~\ref{lem:vjbound} with $\ell=m-2\geq 3$.

\smallskip

\textbf{Step 3}.
We claim that inequality~\eqref{eq:v2term} also holds
for $i=2^{m-2}-1$ and $0\leq j<2^{m-2}-1$.
That is, we are claiming that
\[2^{m-2}+v_2\bigg(\binom{2^{m-2}-1}{j}\bigg) \overset{\textup{?}}{>} j+1 ,\]
which we obtained
by substituting $i=2^{m-2}-1$ and $v_2(A_{m-1},i)=m-3$ from equation~\eqref{eq:Ainduct}
into inequality~\eqref{eq:v2term}, and simplifying.
However, this desired inequality is immediate from our assumption that $j<2^{m-2}-1$.

\smallskip

\textbf{Step 4}.
We claim that inequality~\eqref{eq:v2term} also holds
for $0\leq j \leq i\leq 2^{m-2}-2$.
In that case, the left side of inequality~\eqref{eq:v2term} is
\[ 2^{m-1}-i-j+v_2(A_{m-1,i}) + v_2\bigg(\binom{i}{j}\bigg)
\geq 2^{m-1} -2i + v_2(A_{m-1,i}) > m-1\]
as desired, where the second inequality is by Lemma~\ref{lem:valibound} with $\ell=m-1\geq 4$.

\smallskip

The claims of Steps 1--4, together with equation~\eqref{eq:b_t_formula}, show
that $v_2(B_t) > m-1$ for each $1\leq t \leq 2^{m-1} -3$, which as we noted completes
our proof that $P_{m,1}$ is 2-special.
\end{proof}

\begin{proof}[Proof of Theorem~\ref{thm:special12} for $n=2$]
The main M\"{o}bius product in the definition of $G^{\zeta}_{d,m,n}$ in equation~\eqref{eq:Gdef},
in our case of $G_{m,2}$, is
\[ H_m(c) := \frac{a_{m+1}+a_{m-1}}{a_m+a_{m-1}}
= \frac{a_{m+1}-a_m}{a_m+a_{m-1}} + \frac{a_m+a_{m-1}}{a_m+a_{m-1}} = a_m-a_{m-1}+1, \]
since $a_{m+1}-a_m = (a^2_m +c) - (a^2_{m-1} +c) = a^2_m-a^2_{m-1}$.
Note that $\text{deg}(H_m) = 2^{m-1}$. Inspired by part~(c) of Lemma~\ref{lem:formula_P_{m,n}},
we define
\begin{equation}
\label{eq:Rmdef}
R_m(x):= 4^{2^{m-1}} H_m\bigg(\frac{x-4}{4}\bigg) =
4^{2^{m-1}}\bigg(a_m\bigg(\frac{x-4}{4}\bigg)-a_{m-1}\bigg(\frac{x-4}{4}\bigg)+1\bigg).
\end{equation}

If $m$ is even, then equation~\eqref{eq:Gdef} gives $G_{m,2}=H_m$, and therefore
part~(c) of Lemma~\ref{lem:formula_P_{m,n}} gives $P_{m,2}=R_m$.
On the other hand, if $m$ is odd, then 
\[G_{m,2} = H_m \prod_{d|2}a_d^{-\mu(2/d)} = H_m \cdot \frac{c}{c^2 + c} = (c+1)^{-1} H_m, \]
which is a monic polynomial by \cite[Theorem~2.1]{BG1}, and which has degree degree $2^{m-1}-1$.
Thus, still assuming $m$ is odd, part~(c) of Lemma~\ref{lem:formula_P_{m,n}} gives
\[ P_{m,2}(x) = 4^{(2^{m-1}-1)} G_{m,2}\bigg(\frac{x-4}{4}\bigg)
= \bigg( \frac{1}{4} \bigg(\frac{x-4}{4} + 1\bigg)^{-1} \bigg) R_m(x) = \frac{1}{x}R_m(x), \]
which must be a polynomial (in spite of the $1/x$) because $G_{m,2}$ is a polynomial.
That is, for $m$ odd, the polynomial $R_m$ of equation~\eqref{eq:Rmdef} must have constant term zero.

Thus, it suffices to show, for all $m\geq 2$, that $R_m$ is $2$-special.
Indeed, for $m$ odd, if $R_m$ is 2-special, then since it also has constant term zero,
it is immediate from Definition~\ref{def:special} that $P_{m,2}(x)=R_m(x)/x$ is 2-special as well;
and for $m$ even, we have $P_{m,2}=R_m$.

For $m=2,3,4$, direct computation of the above formulas with Magma gives:
\[ R_2 = x^2-8x+32, \quad R_3 = x^4 - 8x^3 + 128x, \quad\text{and} \]
\[ R_4 = x^8 - 16x^7 + 96x^6 - 128x^5 - 1536x^4 + 8192x^3 - 8192x^2 - 32768x +131072, \]
all of which are $2$-special.
Therefore, for the remainder of the proof, we assume $m\geq 5$, and we must prove that
$R_m$ is 2-special.

Writing
\begin{equation}
\label{eq:Dmidef}
a_m-a_{m-1}+1 = \sum_{i=0}^{2^{m-1}} D_{m,i}c^i
\end{equation}
and using this expansion in equation~\eqref{eq:Rmdef}, we obtain
\begin{align*}
R_m(x) &= 4^{2^{m-1}}\sum_{i=0}^{2^{m-1}} D_{m,i}\bigg(\frac{x-4}{4}\bigg)^i 
= 4^{2^{m-1}} \sum_{i=0}^{2^{m-1}} D_{m,i}\sum_{j=0}^{i} (-1)^{i-j}\binom{i}{j}\bigg(\frac{x}{4}\bigg)^j\\
&= \sum_{i=0}^{2^{m-1}} D_{m,i}\sum_{j=0}^{i} (-1)^{i-j} 4^{2^{m-1}-j}\binom{i}{j}x^j
= x^{2^{m-1}} + \sum_{t=0}^{2^{m-1}-1}E_t x^t,
\end{align*}
where
\begin{equation}
\label{eq:Et}
E_t := 4^{2^{m-1}-t}\sum_{i=t}^{2^{m-1}} (-1)^{i-t}D_{m,i}\binom{i}{t}
\quad\text{for } t=0,1,\dots,2^{m-1}-1.
\end{equation}

Observe that for $i>\deg(a_{m-1})$, the coefficients $D_{m,i}$ in equation~\eqref{eq:Dmidef}
coincide with the coefficients $A_{m,i}$ from equation~\eqref{eq:Alidef}. That is, we have
\begin{equation}
\label{eq:AmBm}
D_{m,i} = A_{m,i} \quad \text{for } i=2^{m-2}+1,\dots,2^{m-1} .
\end{equation}
Recalling from equation~\eqref{eq:Ainduct} that $A_{m,2^{m-1}}=1$ and $A_{m,2^{m-1}-1}=2^{m-2}$,
equations~\eqref{eq:Et} and \eqref{eq:AmBm} yield
\begin{align*}
E_{2^{m-1}-1} &= 4\sum_{i=2^{m-1}-1}^{2^{m-1}}(-1)^{i-2^{m-1}+1}D_{m,i}\binom{i}{2^{m-1}-1}\\
&= 4\sum_{i=2^{m-1}-1}^{2^{m-1}}(-1)^{i-2^{m-1}+1}A_{m,i}\binom{i}{2^{m-1}-1}\\
&= 4A_{m,2^{m-1}-1}-4A_{m,2^{m-1}}2^{m-1} = -2^m.  
\end{align*}
Thus, this coefficient has valuation $v_2(-2^m)=m > v_2(2)$, verifying the first bullet point
of Definition~\ref{def:special}.
Therefore, to finish the proof, we need to verify that
\begin{equation}
\label{eq:vEt}
v_2(E_t)>m \quad \text{for } t=0,1,\dots,2^{m-1}-2 .
\end{equation}
For $0\leq t\leq 2^{m-2}$, equation~\eqref{eq:Et} gives us
\[v_2(E_t)\geq v_2(4^{2^{m-1}-t})\geq 2^{m-1}>m,\]
which establishes the desired inequality, since $m\geq 5$.

For the remainder of the proof of \eqref{eq:vEt}, then, we consider
$t>2^{m-2}$. By equations~\eqref{eq:Et} and~\eqref{eq:AmBm}, it suffices to show
\begin{equation}
\label{eq:tigoal}
2^m-2t+v_2(A_{m,i})+v_2\bigg(\binom{i}{t}\bigg) \overset{\textup{?}}{>} m
\quad \text{ for }2^{m-2} < t \leq 2^{m-1}-2 \,\text{ and }\, i\geq t .
\end{equation}
To prove this, we consider three cases.

\smallskip

\textbf{Case 1}.
If $i=2^{m-1}$, then $A_{m,i}=1$ and $v_2(\binom{i}{t})=m-1-v_2(t)$,
and hence goal~\eqref{eq:tigoal} becomes
\[2^m-2t+m-1-v_2(t) \overset{\textup{?}}{>} m
\quad \text{ for }2^{m-2} < t \leq 2^{m-1}-2. \]
For $t=2^{m-1}-2$, we have  $v_2(t)=1$, and hence
\[2^m-2t+m-1-v_2(t) =2^m - 2^m + 4 + m-2 = m+2 > m, \]
as desired.
Otherwise, we have $t\leq 2^{m-1}-3$, in which case
\[2^m-2t+m-1-v_2(t) \geq 2^{m-1} - t - v_2(t) +m+2 > m+3 > m,\]
where we used the fact that $t\leq 2^{m-1}-3$ and, in the second inequality,
Lemma~\ref{lem:vjbound} with $\ell=m-1 \geq 4$ and $j=t$.

\smallskip

\textbf{Case 2}.
If $i=2^{m-1}-1$, then $A_{m,i}=2^{m-2}$, and hence
\[ 2^m-2t+v_2(A_{m,i})+v_2\bigg(\binom{i}{t}\bigg)
\geq 2^m-2t+m-2 \geq m+2 > m, \]
as desired,
where we used the fact that $t\leq 2^{m-1}-2$ in the second inequality.

\smallskip

\textbf{Case 3}.
It remains to consider $t\leq i\leq 2^{m-1}-2$. In this case, we have
\[ 2^m-2t+v_2(A_{m,i})+v_2\bigg(\binom{i}{t}\bigg) \geq 2^m-2t+v_2(A_{m,i})
\geq 2^m-2i+v_2(A_{m,i}) > m, \]
where the last inequality is by Lemma~\ref{lem:valibound} with $\ell=m$.
\end{proof}

\section{Appendix}
\label{sec:appendix}
As noted in the introduction, the resultants
$\Res(P_{m,n},\Phi_\ell)$ of Conjecture~\ref{conj:resultant_mult_cycl}
are in fact very large integers, according to computational evidence.
Table~\ref{tab:logres} presents the Weil heights of these resultants,
rounded down to the nearest integer,
for $m+n\leq 7$ and $\ell\leq 8$.
(Recall that the Weil height of an integer $N$ is $\log\max\{|N|,1\}$.)

\begin{table}[h]
\caption{Heights of the resultants $\Res(P_{m,n},\Phi_\ell)$}
\label{tab:logres}
\begin{tabular}{|c|c||c|c|c|c|c|c|c|c|}
\hline
$(m,n)$ & $\deg P_{m,n}$ &
\multicolumn{8}{|c|}{$\lfloor \log| \Res(P_{m,n},\Phi_\ell)|\rfloor$ \;\; for\;\; $\ell=$} \\
\cline{3-10}
& & 1 & 2 & 3 & 4 & 5 & 6 & 7 & 8 \\ \hline
(2,1) & 1 & 1 & 1 & 3 & 2 & 5 & 2 & 8 & 5
\\ \hline
(2,2) & 2 & 3 & 3 & 7 & 6 & 11 & 6 & 21 & 13
\\ \hline
(2,3) & 6 & 14 & 14 & 29 & 29 & 58 & 28 & 87 & 58
\\ \hline
(2,4) & 12 & 37 & 37 & 75 & 74 & 150 & 74 & 225 & 150
\\ \hline
(2,5) & 30 & 114 & 115 & 229 & 229 & 458 & 229 & 686 & 457
\\ \hline
(3,1) & 3 & 2 & 2 & 5 & 3 & 11 & 5 & 16 & 11
\\ \hline
(3,2) & 3 & 4 & 4 & 9 & 9 & 19 & 9 & 29 & 19
\\ \hline
(3,3) & 12 & 26 & 27 & 54 & 54 & 109 & 53 & 163 & 108
\\ \hline
(3,4) & 24 & 70 & 71 & 142 & 141 & 283 & 141 & 424 & 283
\\ \hline
(4,1) & 7 & 5 & 5 & 10 & 11 & 22 & 11 & 33 & 22
\\ \hline
(4,2) & 8 & 11 & 11 & 24 & 23 & 47 & 23 & 71 & 47
\\ \hline
(4,3) & 21 & 45 & 45 & 91 & 91 & 183 & 91 & 275 & 183
\\ \hline
(5,1) & 15 & 11 & 10 & 22 & 21 & 45 & 22 & 67 & 45
\\ \hline
(5,2) & 15 & 21 & 21 & 42 & 43 & 86 & 43 & 129 & 86
\\ \hline
(6,1) & 31 & 22 & 22 & 44 & 44 & 86 & 44 & 134 & 88
\\ \hline
\end{tabular}
\end{table}

The data above lead us to propose the following conjecture.
Recall that $a\asymp b$ means that the quantities $a$ and $b$
have the same growth rate, i.e., there are constants $c,C>0$ such that
$ca\leq b \leq Ca$.

\begin{conj}
\label{conj:resgrowth}
Let $m\geq 2$ and $n,\ell\geq 1$. Then
\[ \log\big| \Res\big(P_{m,n}, \Phi_{\ell}\big)\big| \asymp n\deg(P_{m,n}) \deg(\Phi_{\ell}) . \]
\end{conj}

Indeed, throughout Table~\ref{tab:logres}, we have
\[ 0.71  \leq
\frac{ \log\big| \Res\big(P_{m,n}, \Phi_{\ell}\big) \big| }{ n\deg(P_{m,n}) \deg(\Phi_{\ell}) }
\leq 1.44 \]
and outside of the rows with $m+n\leq 4$, the upper bound of $1.44$ drops to $0.82$.

This growth rate should be expected if there are no particular coincidences
aligning the multiplier polynomial $P_{m,n}$ and the cyclotomic polynomial $\Phi_\ell$.
After all, each root $\lambda$ of $P_{m,n}$ is of the form $2^n\alpha$ where
$\alpha=a_n(c_0)\cdots a_{n+m-1}(c_0)$ is an algebraic integer
whose $p$-adic valuation is 0 for all odd primes $p$ and very small for $p=2$.
(See Theorem~1.4 of \cite{BG1}.)
Hence, the expected size of $\lambda-\zeta$ should be about $2^n$, where $\zeta$ is a root of unity.
The resultant is the product of all such differences across all roots
$\lambda$ of $P_{m,n}$ and all roots $\zeta$ of $\Phi_\ell$,
suggesting that $\log| \Res(P_{m,n}, \Phi_{\ell} ) |$ should be
on the order of $n\deg(P_{m,n}) \deg(\Phi_{\ell})$.
Thus, there is both empirical and theoretical evidence to support
Conjecture~\ref{conj:resultant_mult_cycl}, that $|\Res (P_{m,n}, \Phi_{\ell} ) |> 1$.


\bigskip

\textbf{Acknowledgments.}
The first author gratefully acknowledges the support of NSF grant DMS-2101925.\bigskip

\textbf{Data availability statement.} The datasets generated during and/or analysed during the current study are available from the corresponding author on reasonable request.



\begin{thebibliography}{99}

	\bibitem{BD13}
	M.~Baker and L.~DeMarco,
	\emph{Special curves and postcritically-finite polynomials},
	Forum of Mathematics, Pi \textbf{1}, 2013.
	
	\bibitem{BenBook}
	R.L.~Benedetto, 
	\emph{Dynamics in One Non-Archimedean Variable},
	American Mathematical Society, Providence, 2019.
	
	\bibitem{BG1}
	R.L.~Benedetto and V.~Goksel,
	\emph{Misiurewicz polynomials and dynamical units, Part I},
	Preprint, 2021. Available at \texttt{arXiv:2201.07868}
	
	
	\bibitem{BEK19}
	X.~Buff, A.~Epstein, and S.~Koch,
	\emph{Prefixed curves in moduli space}, 
	Amer. J. Math., to appear.
	
	\bibitem{BFKP21}
	X.~Buff, W.~Floyd, S.~Koch, and W.~Parry,
	\emph{Factoring Gleason Polynomials Modulo $2$},
	Journal de Th\'eorie des Nombres de Bordeaux, to appear.
	
	\bibitem{Buff18}
	X.~Buff,
	\emph{On postcritically finite unicritical polynomials},
	New York J. Math. 24, 1111--1122, 2018.
	
	\bibitem{CG}
	L.~Carleson and T.W.~Gamelin,
	\emph{Complex Dynamics},
	Springer-Verlag, New York, 1993.
	
	\bibitem{DH}
	A.~Douady and J.H.~Hubbard,
	\emph{\'{E}tude dynamique des polyn\^{o}mes complexes I \& II},
	Publ.\ Math.\ d'Orsay \textbf{85}, 1984,1985.
	
	\bibitem{Eber99}
	D.~Eberlein,
	\emph{Rational parameter rays of multibrot sets},
	PhD thesis, 
	Technische Universit\"at M\"unchen, 1999.
	
	\bibitem{Epstein12}
	A.~Epstein,
	\emph{Integrality and rigidity for postcritically finite polynomials. With
		an appendix by Epstein and Bjorn Poonen},
	Bull. Lond. Math. Soc. {\bf 44} (2012), no. 1, 39--46.
	
	\bibitem{FG15}
	C.~Favre and T.~Gauthier,
	\emph{Distribution of postcritically finite polynomials},
	Israel J. Math. \textbf{209} (2015), 235--292.
	
	\bibitem{Fakh14}
	N.~Fakhruddin,
	\emph{The algebraic dynamics of generic endomorphisms of $\mathbb{P}^n$},
	Algebra Number Theory \textbf{8} (2014), no. 3, 587-608.
	
	\bibitem{GKNY17}
	D.~Ghioca, H.~Krieger, K.D.~Nguyen, and H.~Ye,
	\emph{The dynamical Andr\'{e}-Oort Conjecture: Unicritical polynomials},
	Duke Math. J. 166(1) (2017), 1--25.
	
	\bibitem{Gok19}
	V.~Goksel,
	\emph{On the orbit of a post-critically finite polynomial of the form $x^d+c$},
	Funct. Approx. Comment. Math. {\bf 62 (1)} (2020), 95--104.
	
	\bibitem{Gok20}
	V.~Goksel,
	\emph{A note on Misiurewicz polynomials},
	Journal de Th\'eorie des Nombres de Bordeaux, Volume 32 (2020), No. 2, p. 373--385.
	
	\bibitem{HT15}
	B.~Hutz, A.~Towsley,
	\emph{Misiurewicz points for polynomial maps and transversality}
	New York J. Math. 21, 297--319, 2015.
	
	\bibitem{Kam88}
	M.~Kaminski,
	\emph{Cyclotomic polynomials and units in cyclotomic number fields},
	J. Number Theory 28.3 (1988): 283--287.
	
	\bibitem{Mil}
	J.~Milnor,
	\emph{Dynamics in One Complex Variable}, 3rd ed.,
	Princeton University Press, Princeton, 2006.
	
	\bibitem{Milnor93}
	J.~Milnor,
	\emph{Geometry and dynamics of quadratic rational maps},
	Experiment. Math. Volume 2, Issue 1 (1993), 37--83.
	
	\bibitem{Milnor09}
	J.~Milnor,
	\emph{Cubic polynomials with periodic critical orbit, Part I},
	In “Complex Dynamics
	Families and Friends”, ed. D. Schleicher, A. K. Peters (2009), 333--411.
	
	\bibitem{Milnor12} 
	J.~Milnor,
	\emph{Arithmetic of unicritical polynomial maps},
	Frontiers in Complex Dynamics:
	In Celebration of John Milnor's 80th Birthday (2012), 15--23.
	
	\bibitem{Poirier93}
	A.~Poirier,
	\emph{On postcritically finite polynomials, part two},
	Preprint Stony Brook IMS (1993).
	
	\bibitem{Washington96}
	L.~Washington,
	\emph{Introduction to Cyclotomic Fields},
	Volume 83 of Graduate Texts Mathematics, 2nd edition, Springer, 1996.
	
	
\end{thebibliography}
\end{document}